\documentclass[a4paper,12pt]{amsart}
\usepackage{tikz}
\usepackage{amssymb}

\numberwithin{equation}{section}

\newtheorem{thm}{Theorem}[section]
\newtheorem{prop}[thm]{Proposition}

\newtheorem{lem}[thm]{Lemma}
\newtheorem{cor}[thm]{Corollary}

\theoremstyle{definition}
\newtheorem{defn}[thm]{Definition}

\theoremstyle{remark}
\newtheorem{rmk}[thm]{Remark}

\newtheorem{ex}[thm]{Example}

\newcommand{\CC}{\mathbb{C}}

\newcommand{\NN}{\mathbb{N}}
\newcommand{\RR}{\mathbb{R}}
\newcommand{\TT}{\mathbb{T}}
\newcommand{\ZZ}{\mathbb{Z}}

\def\Bb{\mathcal{B}}

\def\Tt{\mathcal{T}}

\newcommand{\Tr}{\operatorname{Tr}}
\newcommand{\Aut}{\operatorname{Aut}}

\newcommand{\lsp}{\operatorname{span}}
\newcommand{\clsp}{\overline{\lsp}}

\newcommand{\Per}{\operatorname{Per}}
\newcommand{\id}{\operatorname{id}}

\title[KMS states on the $C^*$-algebra of a higher-rank graph]
{KMS states on the $C^*$-algebra of a higher-rank graph and periodicity in the path space}
\author[an Huef]{Astrid an Huef}
\author[Laca]{Marcelo Laca}
\author[Raeburn]{Iain Raeburn}
\author[Sims]{Aidan Sims}

\address{Astrid an Huef and Iain Raeburn\\
Department of Mathematics and Statistics\\
University of Otago\\
PO Box 56\\
Dunedin 9054\\
New Zealand}
\email{astrid@maths.otago.ac.nz, iraeburn@maths.otago.ac.nz}

\address{Marcelo Laca\\
Department of Mathematics and Statistics\\
University of Victoria\\
Victoria, BC V8W 3P4\\
Canada}
\email{laca@math.uvic.ca}

\address{Aidan  Sims\\
School of Mathematics and Applied Statistics\\
University of Wollongong\\
NSW 2522\\
Australia}
\email{asims@uow.edu.au}

\date{27 April, 2014}

\thanks{This research has been supported by the University of Otago,
the Marsden Fund of the Royal Society of New Zealand,
the Natural Sciences and Engineering Research Council of Canada,
and the Australian Research Council. Part of this work was completed
at the workshops \emph{Graph algebras: Bridges between graph C*-algebras
and Leavitt path algebras} (13w5049) and
\emph{Operator algebras and dynamical systems from number theory} (13w5152)
at the Banff International Research Station.}

\begin{document}

\begin{abstract}
We study the KMS states of the $C^*$-algebra of a strongly connected finite $k$-graph. We
find that there is only one 1-parameter subgroup of the gauge action that can admit a KMS
state. The extreme KMS states for this preferred dynamics are parameterised by the
characters of an abelian group that captures the periodicity in the infinite-path space
of the graph. We deduce that there is a unique KMS state if and only if the $k$-graph
$C^*$-algebra is simple, giving a complete answer to a question of Yang. When the
$k$-graph $C^*$-algebra is not simple, our results reveal a phase change of an unexpected
nature in its Toeplitz extension.
\end{abstract}

\maketitle

\section{Introduction}

Higher-rank graphs ($k$-graphs) are higher-dimensional analogues of directed graphs (the
$1$-graphs). Each $k$-graph $\Lambda$ has a $C^*$-algebra $C^*(\Lambda)$ generated by a
family of partial isometries satisfying relations analogous to the Cuntz-Krieger
relations for a directed graph \cite{KP}. These graphs and their algebras have attracted
a great deal of attention, and the algebras provide illustrative examples for various
active areas of research \cite{PRRS, PRW, PRS, PZ, SZ}. Much of the structure theory of
graph algebras carries over to $k$-graphs, though often with significant changes and
considerable difficulty. It took quite a while, for example, to find a necessary and
sufficient condition for simplicity \cite{RobertsonSims:BLMS07}, and even for $2$-graphs
with a single vertex, this condition is hard to verify \cite{DY}.

Here we study the KMS states for a natural dynamics on the $C^*$-algebra of a $k$-graph.
When a $C^*$-algebra $A$ represents the observables in a physical model, time evolution
is modelled by a continuous action of $\RR$ (a dynamics) on $A$. The equilibrium states
are the states on $A$ that satisfy a commutation relation (the KMS condition) involving a
parameter called the inverse temperature. The KMS condition makes sense for any action of
$\RR$ on any $C^*$-algebra $A$, no matter where $A$ comes from, and the behaviour of the
KMS states always seems to reflect important structural properties of $A$. In recent
years there has been a flurry of activity in which various authors have studied the KMS
states on families of Toeplitz algebras arising in number theory \cite{BC, LR, CDL}, in
the representation theory of self-similar groups \cite{LRRW}, and in graph algebras
\cite{EL, KW, aHLRS1-graphs, aHLRSk-graphs, CL}. These dynamics all manifest a phase
transition in which a simplex of KMS states collapses to a simplex of lower dimension at
a critical inverse temperature.

The $C^*$-algebra of a finite $k$-graph admits a preferred dynamics $\alpha$, which we
discuss at the start of $\S\ref{sec:main}$. Yang has studied this dynamics for $k$-graphs
with a single vertex \cite{Yang:IJM10,Yang:xx11}. She has made a conjecture about the KMS
states and has verified this conjecture for $k=2$ \cite{Yang-preprint2013}. Here we
determine the full simplex of KMS states on $(C^*(\Lambda),\alpha)$ for a large class of
finite $k$-graphs, including all $k$-graphs with one vertex. This allows us to verify
Yang's conjecture for all $k$, and in far greater generality than it was posed. It also
completes the description of the KMS states for the preferred dynamics on the Toeplitz
algebra of $\Lambda$. Many examples exhibit an unexpected phase transition in which the
simplex expands dramatically at the critical inverse temperature instead of collapsing.

\medskip

Our results deal with finite $k$-graphs that are strongly connected in the sense that
there is a directed path from $v$ to $w$ for each pair $v,w$ of vertices. Each $r \in
[0,\infty)^k$ determines a homomorphism of $\RR$ into $\TT^k$, and composing this with
the gauge action on $C^*(\Lambda)$ yields a dynamics $\alpha$. We study KMS states of
$(C^*(\Lambda), \alpha)$. Previous analyses \cite{EFW, KW, aHLRS1-graphs2} for finite
1-graphs depend on Perron-Frobenius theory. So we start in Section~\ref{sec:new PF
matrices} by developing a Perron-Frobenius theory for families of pairwise commuting
non-negative matrices. In Section~\ref{sec:new PF k-graphs} we apply our Perron-Frobenius
theory to the coordinate matrices of $k$-graphs. We characterise the vectors $r$ for
which the associated dynamics admits KMS states on the Toeplitz algebra $\Tt
C^*(\Lambda)$. We show that only one dynamics admits KMS states on $C^*(\Lambda)$. We
call this the preferred dynamics.

Our main result describes the KMS states of $C^*(\Lambda)$ in terms of states on the
$C^*$-algebra of an abelian group $\Per\Lambda$ that captures periodicity in the
infinite-path space $\Lambda^\infty$. We describe $\Per\Lambda$ and its properties in
Section~\ref{sec:per}, and construct in Section~\ref{sec:central rep} an injection $\pi_U
: C^*(\Per\Lambda) \to C^*(\Lambda)$. Our main theorem, Theorem~\ref{thm:main}, says that
the map $\phi \mapsto \phi \circ \pi_U$ is an isomorphism from the KMS simplex of
$C^*(\Lambda)$ to the state space of $C^*(\Per\Lambda)$. (The inverse is described later
in Remark~\ref{rmk:inverse}.) For $k=1$, our theorem recovers the characterisations of
KMS states for Cuntz algebras \cite{OP} and Cuntz-Krieger algebras \cite{EFW}, and we
obtain a new description of the unique KMS state as an integral of vector states.

The proof of our main theorem occupies Sections
\ref{sec:measures}--\ref{sec:surjectivity}. In Section~\ref{sec:measures} we show that
the KMS states of $C^*(\Lambda)$ all induce the same probability measure $M$ on the
spectrum $\Lambda^\infty$ of the diagonal. We deduce in Theorem~\ref{thm:CK states} a
formula for a KMS state $\phi$ in terms of $\phi \circ \pi_U$. In
Section~\ref{sec:surjectivity} we construct a particular KMS state $\phi_1$ of
$C^*(\Lambda)$ as an integral against $M$ of vector states. Unlike for $k=1$ \cite{EFW,
CL}, this KMS state is not always supported on the fixed-point algebra for the gauge
action. Composing $\phi_1$ with gauge automorphisms then yields more KMS states $\phi_z$
(Corollary~\ref{cor:many KMS}). To prove our main theorem, we use Theorem~\ref{thm:CK
states} to see that $\phi \mapsto \phi \circ \pi_U$ is an affine injection, and then
establish surjectivity by showing that every pure state of $C^*(\Per\Lambda)$ has the
form $\phi_z \circ \pi_U$.

In Section~\ref{sec:consequences}, we discuss three applications of our main result.
First, we prove in Theorem~\ref{thm:YangConnection} that $C^*(\Lambda)$ has a unique
gauge-invariant KMS state, and that this KMS state is a factor state if and only if
$\Lambda$ is aperiodic. Restricting this result to $k$-graphs with one vertex confirms
Yang's conjecture in \cite{Yang-preprint2013}. Second, we describe the phase transition
in $\Tt C^*(\Lambda)$ at the critical inverse temperature 1. For many $k$-graphs, the KMS
simplex expands at the critical inverse temperature from a $(|\Lambda^0| -
1)$-dimensional simplex to an infinite-dimensional simplex. Third, we show that the KMS
simplex of $C^*(\Lambda)$ is highly symmetric: it carries a free and transitive action of
$(\Per\Lambda)\widehat{\;}$. We conclude in Section~\ref{sec:groupoids} by relating our
results to Neshveyev's analysis of KMS states on groupoid $C^*$-algebras~\cite{N}.

\section{Background}

\subsection{Higher-rank graphs}
A higher-rank graph of rank $k$, or $k$-graph, is a countable category $\Lambda$ equipped
with a functor $d : \Lambda \to \NN^k$ satisfying the factorisation property: whenever
$d(\lambda) = m + n$ there exist unique $\mu \in d^{-1}(m)$ and $\nu \in d^{-1}(n)$ such
that $\lambda = \mu\nu$.  We denote $d^{-1}(n)$ by $\Lambda^n$. The elements of
$\Lambda^0$ are precisely the identity morphisms, and we call them vertices.  We refer to
other elements of $\Lambda$ as paths. We write $r,s : \Lambda \to \Lambda^0$ for the maps
determined by the domain and codomain maps in $\Lambda$. We assume that $\Lambda^{e_i}
\not= \emptyset$ for each generator $e_i$ of $\NN^k$ (otherwise we can just regard
$\Lambda$ as a $(k-1)$-graph).

If $\lambda = \mu\nu\tau\in\Lambda$ where $d(\mu) = m$, $d(\nu) = n - m$ and $d(\tau) =
d(\lambda) - n$, then we denote $\nu = \lambda(m, m+n)$. We use the convention that for
$\lambda \in \Lambda$ and $X \subseteq \Lambda$,
\[
\lambda X = \{\lambda\mu : \mu \in X\text{ and } r(\mu) = s(\lambda)\},
\quad X\lambda = \{\mu\lambda : \mu \in X\text{ and }s(\mu) = r(\lambda)\},
\]
and so forth. We write $\Lambda^{\min}(\mu,\nu)$ for the set $\{(\alpha,\beta) :
\mu\alpha = \nu\beta \in \Lambda^{d(\mu) \vee d(\nu)}\}$.

We say that $\Lambda$ is \emph{finite} if $\Lambda^n$ is finite for all $n \in \NN^k$ and
say it has \emph{no sources} if $v \Lambda^{e_i} \not= \emptyset$ for all $v \in
\Lambda^0$ and all $e_i$; it follows that $v\Lambda^n \not= \emptyset$ for all $v \in
\Lambda^0$ and $n \in \NN^k$. We say that $\Lambda$ is \emph{strongly connected} if, for
all $v,w \in \Lambda^0$, the set $v\Lambda w$ is nonempty.

\begin{lem}\label{lem-nosources}
Let $\Lambda$ be a strongly connected  $k$-graph. Then
\begin{enumerate}
\item\label{lem-nosources-item1} $\Lambda$ has no sources and
\item\label{lem-nosources-item2}  for all $v\in\Lambda^0$ and $n\in\NN^k$, the set
    $\Lambda^nv\neq\emptyset$.
\end{enumerate}
\end{lem}
\begin{proof}
For \eqref{lem-nosources-item1}, let $v\in\Lambda^0$ and $i\in\{1,\dots,k\}$.  By our
convention, $\Lambda^{e_i}\neq \emptyset$. Let $\lambda\in \Lambda^{e_i}$. Since
$\Lambda$ is strongly connected, there exists $\mu\in v\Lambda r(\lambda)$.  By the
factorisation property, $\mu\lambda=\lambda'\mu'$ for some $\lambda'\in v\Lambda^{e_i}$.
Thus $v\Lambda^{e_i}\neq\emptyset$. This gives \eqref{lem-nosources-item1}. The proof of
\eqref{lem-nosources-item2} is similar.
\end{proof}

\subsection{Higher-rank graph $C^*$-algebras}

Let $\Lambda$ be a finite $k$-graph with no sources. A Cuntz-Krieger $\Lambda$-family is
a collection $\{t_\lambda : \lambda \in \Lambda\}$ of partial isometries  in a
$C^*$-algebra $A$ such that
\begin{itemize}
\item[(CK1)] the elements $\{t_v : v \in \Lambda^0\}$ are mutually orthogonal
    projections,
\item[(CK2)] $t_\mu t_\nu = t_{\mu\nu}$ when $s(\mu) = r(\nu)$,
\item[(CK3)] $t^*_\mu t_\mu = t_{s(\mu)}$ for all $\mu$, and
\item[(CK4)] for all $v \in \Lambda^0$ and $n \in \NN^k$, we have $t_v =
    \sum_{\lambda \in v\Lambda^n} t_\lambda t^*_\lambda$.
\end{itemize}
The $C^*$-algebra $C^*(\Lambda)$ of $\Lambda$ is generated by a universal Cuntz-Krieger
$\Lambda$-family $\{s_\lambda\}$. We write $p_v := s_v$ for $v \in \Lambda^0$. The
Cuntz-Krieger relations imply that for all $\mu,\nu\in\Lambda$
\[
s^*_\mu s_\nu = \sum_{(\alpha,\beta) \in \Lambda^{\min}(\mu,\nu)} s_\alpha s^*_\beta
\]
(we interpret empty sums as zero). In particular, if $d(\mu)=d(\nu)$, then
\[
s^*_\mu s_\nu = \delta_{\mu, \nu}p_{s(\mu)}.
\]
Relations (CK1)~and~(CK2) then imply that $C^*(\Lambda) = \clsp\{s_\mu s^*_\nu : \mu,\nu
\in \Lambda, s(\mu) = s(\nu)\}$. There is a strongly continuous action
$\gamma:\TT^k\to\Aut C^*(\Lambda)$ such that $\gamma_z(p_v)=p_v$ and
$\gamma_z(s_\lambda)=z^{d(\lambda)}s_\lambda$ for $z\in\TT^k$.  This action is called the
\emph{gauge action}.

\subsection{The Perron-Frobenius theorem}

There are several Perron-Frobenius theorems; the one we use here applies to irreducible
matrices. Let $S$ be a finite set. We say a matrix  $A\in M_S(\CC)$  is
\emph{non-negative} if $A(s,t) \geq 0$ for all $s,t\in S$ and is \emph{positive} if
$A(s,t)> 0$ for all $s,t\in S$. A non-negative matrix $A\in M_S$ is \emph{irreducible} if
for each $s,t\in S$ there exists $N\in\NN$ such that $A^N(s,t)>0$. Equivalently, $A$ is
irreducible if there is a finite subset $F \subseteq \NN$ such that $\sum_{n \in F} A^n$
is positive.

Let $A$ be an irreducible matrix. The Perron-Frobenius theorem (see, for example,
\cite[Theoren~1.5]{Seneta}) says that the spectral radius $\rho(A)$ is an eigenvalue of
$A$ with a $1$-dimensional eigenspace and a  positive eigenvector; we call the unique
positive eigenvector with eigenvalue $\rho(A)$ and unit $1$-norm  the \emph{unimodular
Perron-Frobenius eigenvector of $A$}.

\section{Perron-Frobenius theory for commuting matrices}\label{sec:new PF matrices}

In \cite{aHLRSk-graphs}, we employed a version of the Perron-Frobenius theorem for
pairwise commuting irreducible matrices \cite[Lemma~2.1]{aHLRSk-graphs} to describe KMS
states on the $C^*$-algebras of coordinatewise-irreducible $k$-graphs. David Pask
subsequently pointed out to us that he and Kumjian had adapted a technique from Putnam
\cite{Putnam:notes} to prove a Perron-Frobenius theorem for strongly connected finite
$k$-graphs in \cite[Lemma~4.1]{KumjianPask:ETDS03}. In this section, we adapt Kumjian and
Pask's ideas to formulate a Perron-Frobenius theorem for families of commuting
non-negative matrices that are jointly irreducible in an appropriate sense. Our primary
use for this theorem is in the context of finite $k$-graphs, and we deduce what we need
to know about these in the next section. But our results are applicable to more general
classes of matrices than those arising from $k$-graphs and may be of independent
interest.

Let $S$ be a finite set, $\{A_1, \dots, A_k\} \subseteq M_S\big([0, \infty)\big)$ a
family of commuting matrices, $n=(n_1,\dots, n_k) \in \NN^k$ and $F$ a finite subset of
$\NN^k$.  We use the multi-index notation
\[
	A^n := \prod_{i=1}^k A_i^{n_i}\quad\text{and}\quad A_F := \sum_{n \in F} A^n.
\]
We say that the family $\{A_1, \dots, A_k\}$ is \emph{irreducible} if each $A_i \not= 0$
and there exists a finite subset $F \subseteq \NN^k$ such that $A_F(s,t) > 0$ for all
$s,t \in S$; that is, $A_F$ is positive. Observe that in an irreducible family of
matrices, the individual $A_i$ may not be irreducible. So an irreducible family of
matrices is not the same thing as a family of irreducible matrices. For examples of this
distinction arising from $k$-graphs, see Example~\ref{ex:irreducible family}.

\begin{prop}\label{prp:newer PF}
Suppose that $\{A_1, \dots, A_k\}$ is an irreducible family of nonzero commuting matrices
in $M_S\big([0, \infty)\big)$.  Let $F$ be a finite subset of $\NN^k$ such that
$A_F(s,t)>0$ for all $s,t\in S$ and let $x$ be the unimodular Perron-Frobenius
eigenvector of $A_F$.
\begin{enumerate}
\item\label{it:PFev-newer}
\begin{enumerate}
\item \label{item-i-newer} The vector $x$ is the unique non-negative vector of
    unit $1$-norm that is a common eigenvector of all the $A_i$.
\item\label{item-0-newer} We have $A_ix=\rho(A_i)x$ for each $i$, and each
    $\rho(A_i)>0$.
\item\label{item-ii-newer} If $z \in \CC^S$ and $A_i z = \rho(A_i) z$ for all
    $1\leq i\leq k$, then $z \in \CC x$.
\end{enumerate}
\item\label{it:subinv-newer} Suppose that $y \in [0,\infty)^S$ is non-zero and
    $\lambda \in [0,\infty)^k$ satisfies $A_i y \le \lambda_i y$ for all $1\leq i\leq
    k$.
    \begin{enumerate}
    \item\label{item:subinv-newer1} Then $y > 0$ in the sense that each $y_s >
        0$, and $\lambda_i \ge \rho(A_i)$ for all $1\leq i \leq k$.
    \item\label{item:subinv-newer2} If $\lambda_i=\rho(A_i)$ for all $1\leq i\leq
        k$ and $y$ has unit $1$-norm, then $y=x$.
  \end{enumerate}
\item\label{prp:newer PF-item3} Let $n \in \NN^k$. Then $\rho(A^n)=\prod_{i=1}^k
    \rho(A_i)^{n_i}  > 0$.
\end{enumerate}
\end{prop}

The following lemma helps in the proof of Proposition~\ref{prp:newer PF}.

\begin{lem}\label{lem:nonirred subinvariance}
Let $B \in M_S\big([0, \infty)\big)$, and suppose that $x \in (0,\infty)^S$ and  $\lambda
\ge 0$ satisfy $Bx \le \lambda x$. Then $\lambda \ge \rho(B)$.
\end{lem}
\begin{proof}
Choose a sequence $\{B_j\}$ in $M_S\big((0,\infty)\big)$ converging to $B$. Then $B_j x
\to B x \le \lambda x$. Fix $\epsilon>0$. The entries of $x$ are strictly positive, and
so $B_j x < (\lambda + \epsilon) x$ for large $j$. Part~(b) of the Subinvariance Theorem
\cite[Theorem~1.6]{Seneta} for the positive matrix $B_j$ gives $\lambda + \epsilon \ge
\rho(B_j)$ for large $j$. Since the eigenvalues of a complex matrix vary continuously
with its entries (see, for example, \cite[Theorem~B]{HarrisMartin:PAMS87}), we have
$\rho(B_j) \to \rho(B)$ as $j \to \infty$. Hence $\lambda + \epsilon \ge \rho(B)$. Thus
$\lambda \ge \rho(B)$.
\end{proof}

\begin{proof}[Proof of Proposition~\ref{prp:newer PF}]
Such a finite set $F$ exists because $\{A_1, \dots, A_k\}$ is irreducible.

(\ref{it:PFev-newer}) Since $x$ is a Perron-Frobenius eigenvector, $x > 0$ by
\cite[Theorem~1.5 (b) and (f)]{Seneta}. Let $i\in\{1,\dots,k\}$.

We have
\[
A_F (A_i x) =  A_i (A_F x) = \rho(A_F) A_i x.
\]
So $A_i x$ is a non-negative eigenvector for $A_F$ with eigenvalue $\rho(A_F)$. Since the
eigenspace corresponding to $\rho(A_F)$ is one-dimensional (\cite[Theorem~1.5
(f)]{Seneta}) we have $A_ix = \lambda_i x$ for some $\lambda_i \in [0,\infty)$. To prove
uniqueness, we first claim that
\begin{equation}\label{eq:prodsumeval}
y > 0\text{ and }A_i y = \eta_i y\text{ for all $i$}
    \qquad\Longrightarrow\qquad
\sum_{n \in F} \prod_i \eta_i^{n_i} = \rho(A_F).
\end{equation}
To see this, suppose that $y > 0$ and $A_i y = \eta_i y$ for all $i$. We have
\begin{equation}\label{eq:AFy}
    A_F y = \Big(\sum_{n \in F} \prod_i A_i^{n_i}\Big) y = \Big(\sum_{n \in F} \prod_i \eta_i^{n_i}\Big) y.
\end{equation}
Thus $y$ is an eigenvector of $A_F$ with eigenvalue $\eta := \sum_{n\in
F}\prod_{i=1}^k\eta_i^{n_i}$. Since $\|y\|_1 = 1$, some $y_s > 0$. Since $A_F$ is
positive, we have $A_F(s,s) > 0$ and so $\eta y_s = (A_F y)_s \ge A_F(s,s)y_s > 0$. So
$\eta > 0$. Since $y\geq 0$ and $y\neq 0$, the ``if'' direction of the last sentence of
the Subinvariance Theorem \cite[Theorem~1.6]{Seneta} gives $\eta = \rho(A_F)$.

Now suppose that $y$ is a nonnegative unimodular common eigenvector of the $A_i$. Then
\eqref{eq:prodsumeval}~and~\eqref{eq:AFy} show that $x$ and $y$ are non-negative and of
the same norm in the same one-dimensional eigenspace of $A_F$, hence are equal. This
completes the proof of~(\ref{item-i-newer}).

Since the $A_i$ and $x$ are real and non-negative, and by definition of the spectral
radius, each $0\leq \lambda_i\leq \rho(A_i)$. Lemma~\ref{lem:nonirred subinvariance}
(applied to $A_i$, $x$ and $\lambda_i$) implies that $\lambda_i\geq \rho(A_i)$ as well,
giving $\lambda_i=\rho(A_i)$.  Thus $A_ix = \rho(A_i)x$ for each $i$. Since $x$ is
positive and $A_i \not= 0$ this forces each $\rho(A_i) > 0$. This
proves~(\ref{item-0-newer}).

The claim~\eqref{eq:prodsumeval} applied with $y = x$ and $\eta_i = \rho(A_i)$ now gives
\begin{equation}\label{eq:e-value}
\rho(A_F) = \sum_{n \in F} \prod_i\rho(A_i)^{n_i}.
\end{equation}

For (\ref{item-ii-newer}) suppose that $A_i z = \rho(A_i) z$ for $1\le i \le k$. Then
\[
    A_F z = \Big(\sum_{n \in F} \prod_i \rho(A_i)^{n_i}\Big) z=\rho(A_F)z
\]
using \eqref{eq:e-value}. Thus $z$ is an eigenvector of $A_F$ with eigenvalue
$\rho(A_F)$. The eigenspace of the Perron-Frobenius eigenvalue $\rho(A_F)$ is
one-dimensional,  and hence $z \in \CC x$.

(\ref{it:subinv-newer}) Fix $s \in S$. Since $y \not= 0$, there exists $t \in S$ such
that $y_t > 0$. Since $\{A_1, \dots, A_k\}$ is an irreducible family, there exists $n \in
\NN^k$ such that $A^n(s,t) > 0$. Then
\[
    \lambda^n y_s = \Big(\prod_i\lambda_i^{n_i}\Big) y_s \ge (A^n y)_s \ge A^n(s,t) y_t > 0.
\]
Thus $y_s>0$ for all $s\in S$, so $y>0$. Next, fix $i$. By assumption, $\lambda_i\geq 0$
and $A_i y \le \lambda_i y$. Thus Lemma~\ref{lem:nonirred subinvariance} applied to
$A_i$, $\lambda_i$ and $y$ gives $\lambda_i \ge \rho(A_i)$. This establishes
\eqref{item:subinv-newer1}.

Next, suppose that $A_iy\leq \rho(A_i)y$ for $1\leq i\leq k$.  Then
\[
    A_F y \leq \Big(\sum_{n \in F} \prod_i \rho(A_i)^{n_i}\Big) y=\rho(A_F)y
\]
using \eqref{eq:e-value}. So the ``only-if'' direction of the last sentence of the
Subinvariance Theorem \cite[Theorem~1.6]{Seneta} says that $A_F y=\rho(A_F)y$. Now $x$
and $y$ are non-negative of the same norm in the same one-dimensional eigenspace, hence
are equal. This gives \eqref{item:subinv-newer2}.

\eqref{prp:newer PF-item3} By \eqref{item-0-newer}, $x$ is a common eigenvector of the
$A_i$ with eigenvalue $\rho(A_i)$. Thus
\[
     A^n x = \Big(\prod_i A_i^{n_i}\Big) x = \Big(\prod_i \rho(A_i)^{n_i}\Big) x.
\]
and hence $\prod_i \rho(A_i)^{n_i}\leq \rho(A^n)$. Since $x>0$,  Lemma~\ref{lem:nonirred
subinvariance} implies that  $\prod_i \rho(A_i)^{n_i}\geq \rho(A^n)$. Now
$\rho(A^n)=\prod_i \rho(A_i)^{n_i}>0$ because each $\rho(A_i) > 0$ by
(\ref{item-0-newer}).
\end{proof}

\section{KMS states on Toeplitz algebras of strongly connected $k$-graphs}\label{sec:new PF k-graphs}
In this section, we apply the results of Section~\ref{sec:new PF matrices} to the
coordinate matrices of finite $k$-graphs. We use them to improve the results of
\cite{aHLRSk-graphs} about which dynamics on the Toeplitz algebra of a
coordinatewise-irreducible $k$-graph admit KMS states. We will also use the results of
this section extensively later to characterise the KMS states on the Cuntz-Krieger
algebras of finite $k$-graphs.

Let $\Lambda$ be a finite $k$-graph.  For $1\leq i\leq k$, let $A_i$ be the matrix in
$M_{\Lambda^0}$ with entries $A_i(v,w) = |v\Lambda^{e_i}w|$ for $v, w\in\Lambda^0$. We
call the $A_i$ the \emph{coordinate matrices} of $\Lambda$.

\begin{lem}\label{lem:s.c<->irred}
Let $\Lambda$ be a finite $k$-graph with coordinate matrices $A_1, \dots, A_k$. Then the
$A_i$ are nonzero pairwise-commuting matrices, and $\Lambda$ is strongly connected if and
only if $\{A_1, \dots, A_k\}$ is an irreducible family of matrices.
\end{lem}
\begin{proof}
The $A_i$ are nonzero by our convention that each $\Lambda^{e_i}$ is nonempty. The
factorisation property of $\Lambda$ ensures that
\[
(A_i A_j)(v, w) = |v\Lambda^{e_i + e_j} w| = (A_j A_i)(v, w),
\]
so the $A_i$ commute.

Suppose that $\Lambda$ is strongly connected. Let $v,w \in \Lambda^0$. There exists
$n_{v,w} \in \NN^k$ such that $v \Lambda^{n_{v,w}} w \not= \emptyset$ because $\Lambda$
is strongly connected. Now $F := \{n_{v,w} : v,w \in \Lambda^0\}$ satisfies $A_F(v,w) \ge
A_{n_{v,w}}(v,w) > 0$ for all $v,w$. Thus $\{A_1, \dots, A_k\}$ is an irreducible family.

Now suppose that $\{A_1, \dots, A_k\}$ is an irreducible family. Choose $F$ such that
$A_F$ is positive. For $v,w \in \Lambda^0$, we have $A_F(v,w) \not= 0$ and so there
exists $n \in F$ such that $|v\Lambda^n w| = A^n(v,w) \not= 0$. So $\Lambda$ is strongly
connected.
\end{proof}

The next corollary sums up how we will use the results of Section~\ref{sec:new PF
matrices}.

\begin{cor}\label{cor-PFgraph}
Let $\Lambda$ be a strongly connected finite $k$-graph. For $1\le i \le k$, let $A_i \in
M_{\Lambda^0}\big([0, \infty)\big)$ be the matrix with entries $A_i(v,w) = |v
\Lambda^{e_i} w|$.
\begin{enumerate}
\item\label{cor-PFgraph0} Each $\rho(A_i)>0$, and for $n\in\NN^k$, we have
    $\rho(A^n)=\prod_i \rho(A_i)^{n_i}>0$.
\item\label{cor-PFgraph1} There exists a unique non-negative vector
    $x^\Lambda\in[0,\infty)^{\Lambda^0}$ with unit $1$-norm such that
    $A_ix^\Lambda=\rho(A_i)x^\Lambda$ for all $1\leq i\leq k$. Moreover,
    $x^\Lambda>0$ in the sense that $x^\Lambda_v>0$ for all $v\in\Lambda^0$.
\item\label{cor-PFgraph2}  If $z \in \CC^{\Lambda^0}$ and $A_i z = \rho(A_i) z$ for
    all $1\leq i\leq k$, then $z \in \CC x^\Lambda$.
\item\label{cor-PFgraph3}  If $y\in [0,\infty)^{\Lambda^0}$ has unit $1$-norm and
    $A_iy\leq \rho(A_i)y$ for all $1\leq i\leq k$, then $y=x^\Lambda$.
\end{enumerate}
\end{cor}
\begin{proof}
Lemma~\ref{lem:s.c<->irred} shows that the $A_i$ are an irreducible family. So
\eqref{cor-PFgraph0} is immediate from parts \eqref{item-0-newer}~and~\eqref{prp:newer
PF-item3} of Proposition~\ref{prp:newer PF}. By Proposition~\ref{prp:newer
PF}\eqref{item-i-newer}, the unimodular Perron-Frobenius eigenvector $x^\Lambda$ of $A_F$
is the unique non-negative, unimodular common eigenvector of the $A_i$. Since $x^\Lambda$
is a Perron-Frobenius eigenvector for an irreducible matrix, $x^\Lambda>0$. This gives
\eqref{cor-PFgraph1}. Parts \eqref{cor-PFgraph2}~and~\eqref{cor-PFgraph3} follow from
parts \eqref{item-ii-newer}~and~\eqref{item:subinv-newer2} of Proposition~\ref{prp:newer
PF} respectively.
\end{proof}

\begin{ex}\label{ex:irreducible family}
Corollary~\ref{cor-PFgraph} is an improvement on \cite[Proposition~7.1]{aHLRSk-graphs},
which applies to finite $k$-graphs that are coordinatewise irreducible in the sense that
each $A_i$ is an irreducible matrix. To see that there are many strongly connected
$k$-graphs that are not coordinatewise irreducible, consider strongly connected
$1$-graphs $E,F$ and suppose that $F$ has at least two vertices. Let $\Lambda$ be the
cartesian-product $2$-graph $\Lambda = E \times F$. The connected components of the
coordinate graph $(\Lambda^0, \Lambda^{(1,0)}, r, s)$ are the sets
\[
E^0 \times \{v\}, \quad v \in F^0.
\]
So $A_1$ is block-diagonal with blocks indexed by $F^0$, and in particular $\Lambda$ is
not coordinatewise irreducible. But it is strongly connected: take $(u_1,v_1), (u_2, v_2)
\in E^0 \times F^0$ and use that $E$ and $F$ are strongly connected to find $\mu \in u_1
E^* u_2$ and $\nu \in v_1 F^* v_2$; then $(\mu,\nu) \in (u_1, v_1)\Lambda (u_2, v_2)$.
\end{ex}

\begin{defn}\label{PF-vector}
Let $\Lambda$ be a strongly connected finite $k$-graph. We call the vector $x^\Lambda$ of
Corollary~\ref{cor-PFgraph} the \emph{unimodular Perron-Frobenius eigenvector} of
$\Lambda$.
\end{defn}

We write $\rho(\Lambda)$ for the vector $\big(\rho(A_i)\big) \in [0,\infty)^k$, and
$\ln\rho(\Lambda)$ for the vector $\big(\ln\rho(A_i)\big) \in [-\infty, \infty)^k$. For
$n\in\NN^k$ we have $A^n x^\Lambda = \rho(\Lambda)^nx^\Lambda$ where $\rho(\Lambda)^n :=
\prod_{i=1}^k\rho(A_i)^{n_i}$ is defined using multi-index notation.

\begin{rmk}
At first glance, Corollary~\ref{cor-PFgraph}, which allows Definition~\ref{PF-vector},
appears very similar to Proposition~7.1 of \cite{aHLRSk-graphs} except that it has a
weaker hypothesis. In Proposition~\ref{prp:newer PF}, however, the individual $A_i$ need
not be irreducible, and so it does not make sense to discuss ``the unique unimodular
Perron-Frobenius eigenvectors of the $A_i$." In \cite[Proposition~7.1]{aHLRSk-graphs},
the $A_i$ and $A^n$ are irreducible, so they each have a unique unimodular
Perron-Frobenius eigenvector; the proposition asserts that these eigenvectors are all
equal. When we apply Proposition~\ref{prp:newer PF} to the family of coordinate matrices
$A_i$ of a strongly connected graph, each $A_i$ may have multiple linearly independent
non-negative eigenvectors. The result asserts that there is a unique non-negative
unimodular eigenvector $x^\Lambda$ common to all the $A_i$, and that the spectral radius
of each $A^n$ is achieved at $x^\Lambda$.
\end{rmk}

We finish the section by using Proposition~\ref{prp:newer PF}  to strengthen
Corollaries~4.3~and~4.4 of~\cite{aHLRSk-graphs}. Recall that a Toeplitz-Cuntz-Krieger
$\Lambda$-family consists of   partial isometries $\{ T_\lambda: \lambda \in \Lambda\}$
satisfying (CK1)--(CK3) and the additional relations
\begin{itemize}
\item [(T4)] $T_v \ge \sum_{\lambda \in v\Lambda^n} T_\lambda T^*_\lambda$ for all $v
    \in \Lambda^0$ and $n \in \NN^k$; and
\item [(T5)] $T_\mu^* T_\nu = \sum_{(\alpha, \beta) \in \Lambda^{\min}(\mu,\nu)}
    T_\alpha T^*_\beta$ for all $\mu,\nu$, where empty sums are interpreted as zero.
\end{itemize}
The Toeplitz algebra $\Tt C^*(\Lambda)$ of $\Lambda$ is generated by a universal
Toeplitz-Cuntz-Krieger $\Lambda$-family $\{t_\lambda\}$. We write $q_v := t_v$ for $v \in
\Lambda^0$.

There is a strongly continuous action $\gamma:\TT^k\to\Aut \Tt C^*(\Lambda)$ such that
$\gamma_z(q_v)=q_v$ and $\gamma_z(t_\lambda)=z^{d(\lambda)}t_\lambda$ for $z\in\TT^k$.
This action is called the \emph{gauge action}.  We use the same letter $\gamma$ for the
gauge actions on $C^*(\Lambda)$ and  $\Tt C^*(\Lambda)$; this is safe because the
quotient map of $\Tt C^*(\Lambda)$ onto $C^*(\Lambda)$ intertwines the two.

\begin{cor}\label{cor:aHLRS2-Cor4.3,4.4}
Suppose that $\Lambda$ is a strongly connected finite $k$-graph. Let $\beta\in
[0,\infty)$. Fix $r \in \RR^k$ and define $\alpha : \RR \to \Aut\Tt C^*(\Lambda)$ by
$\alpha_t = \gamma_{e^{itr}}$.
\begin{enumerate}
\item\label{it:KMSToeplitz} There exists a KMS$_\beta$ state for $(\Tt C^*(\Lambda),
    \alpha)$ if and only if $\beta r \ge \ln \rho(\Lambda)$.
\item\label{it:KMSCK} If there is a KMS$_\beta$ state for $(\Tt C^*(\Lambda),
    \alpha)$ that factors through $C^*(\Lambda)$, then $\beta r = \ln \rho(\Lambda)$.
\item\label{it:allFactor} If $\beta r = \ln\rho(\Lambda)$, then every KMS$_\beta$
    state for $(\Tt C^*(\Lambda), \alpha)$ factors through $C^*(\Lambda)$.
\end{enumerate}
\end{cor}
\begin{proof}
(\ref{it:KMSToeplitz}) First suppose that $\phi$ is a KMS$_\beta$ state.  Let $v
\in\Lambda^0$ and set $m_v^\phi:=\phi(q_v)$.  For $1\le i \le k$, relation~(T4) and the
KMS condition give
\begin{align}
0
&\le \phi\Big(q_v - \sum_{\lambda \in v\Lambda^{e_i}} t_\lambda t^*_\lambda\Big)
    = \phi(q_v) - \sum_{w \in \Lambda^0} |v\Lambda^{e_i} w| e^{-\beta r_i} \phi(t^*_\lambda t_\lambda)
        \label{eq:delta value} \\
&=\phi(q_v)  - e^{-\beta r_i} \sum_{w \in \Lambda^0} A(v,w)\phi(t_{s(\lambda)})
    = \big(m^\phi - e^{-\beta r_i} A_i m^\phi\big)_v.\notag
\end{align}
Hence $A_i m^\phi \le e^{\beta r_i} m^\phi$ for each $i$. Now Proposition~\ref{prp:newer
PF}(\ref{it:subinv-newer}), applied to $A_i$, $e^{\beta r_i}$ and $m^\phi$,  implies that
each $e^{\beta r_i} \ge \rho(A_i)$. Thus $\beta r\ge \ln\rho(\Lambda)$.

Second, suppose that $\beta r\ge \ln\rho(\Lambda)$. Choose a sequence $\{r_n\}$ in
$\RR^k$ converging to $r$ from above and a sequence $\beta_n$ converging to $\beta$ from
above. So each $\beta_n r_n > \ln\rho(\Lambda)$. For each $n$ let $\alpha^{r_n}$ be the
dynamics $\alpha^{r_n}_t = \gamma_{e^{itr_n}}$. By \cite[Theorem~6.1]{aHLRSk-graphs}
there exists, for each $n$, a KMS$_{\beta_n}$ state $\phi_n$ of $(\Tt C^*(\Lambda),
\alpha^{r_n})$. We have $\alpha^{r_n}_t(t_\mu t^*_\nu) \to \alpha_t(t_\mu t^*_\nu)$ for
all $\mu,\nu$, and so an $\varepsilon/3$ argument shows that $\|\alpha^{r_n}_t(a) -
\alpha_t(a)\| \to 0$ for all $a$. Now \cite[Proposition~5.3.25]{BR} shows that $(\Tt
C^*(\Lambda), \alpha)$ has a KMS$_\beta$ state.

(\ref{it:KMSCK}) Suppose that $\phi$ is a KMS$_\beta$ state of $(\Tt C^*(\Lambda),
\alpha)$ that factors through $C^*(\Lambda)$. Then we have equality in~\eqref{eq:delta
value}, and so $m^\phi$ is a unimodular non-negative eigenvector of each $A_i$ with
eigenvalue $e^{\beta r_i}$. Thus Proposition~\ref{prp:newer
PF}(\ref{item-i-newer})~and~(\ref{item-0-newer}) imply that $e^{\beta r_i} = \rho(A_i)$
for each $i$. Thus $\beta r=\ln\rho(\Lambda)$.

(\ref{it:allFactor}) Suppose that $\beta r = \ln\rho(\Lambda)$ and that $\phi$ is a
KMS$_\beta$ state of $(\Tt C^*(\Lambda), \alpha)$. Then~\eqref{eq:delta value} shows that
$\rho(\Lambda)_i m^\phi \ge A_i m^\phi$ for each $i$. Now
Corollary~\ref{cor-PFgraph}(\ref{cor-PFgraph3}) implies that $m^\phi = x^\Lambda$, and
hence $e^{\beta r_i} m^\phi = \rho(\Lambda)_i m^\phi = A_i m^\phi$ for all $i$. Hence
\cite[Proposition~4.1(b)]{aHLRSk-graphs} implies that $\phi$ factors through
$C^*(\Lambda)$.
\end{proof}

\begin{rmk} From the point of view developed by Bratteli, Elliott and Kishimoto \cite{BEK}, the
collection $\operatorname{Lie}(\TT^k)$ of continuous homomorphisms from $\RR$ to $\TT^k$
is the collection of possible finite inverse temperatures for KMS states for the gauge
action $\gamma$. A KMS state for $\gamma$ at inverse temperature $\beta \in
\operatorname{Lie}(\TT^k)$ is then a KMS$_1$ state for the action $\gamma \circ \beta$ of
$\RR$.

Embed $\RR^k$ in $\operatorname{Lie}(\TT^k)$ via $\beta \mapsto (t \mapsto e^{i\beta
t})$. Corollary~\ref{cor:aHLRS2-Cor4.3,4.4}(\ref{it:KMSToeplitz}) says that the gauge
action on $\Tt C^*(\Lambda)$ admits a KMS state at inverse temperature $\beta \in \RR^k$
if and only if $\beta \in \big[\ln\rho(A_1), \infty\big) \times \cdots \times
\big[\ln\rho(A_k), \infty\big)$. Corollary~\ref{cor:aHLRS2-Cor4.3,4.4}(\ref{it:KMSCK})
says that the KMS states that factor through $C^*(\Lambda)$ are those at inverse
temperature $\ln\rho(\Lambda)$. So from the point of view of \cite{BEK},
Corollary~\ref{cor:aHLRS2-Cor4.3,4.4} identifies $\beta = \ln\rho(\Lambda)$ as the
critical inverse temperature for~$\gamma$.
\end{rmk}

\section{Periodicity of $k$-graphs}\label{sec:per}

In this section we describe the  periodicity group of a strongly connected finite
$k$-graph $\Lambda$. This group is a key ingredient in the statement of our main theorem.
The fundamental idea behind our analysis involves source- and range-preserving bijections
between certain sets of paths, and comes from Davidson and Yang's analysis of periodicity
in 2-graphs with one vertex~\cite{DY}.

Our results in this section and the next also follow from the more general results of
\cite{CarlsenKangEtAl:xxxx} (see also \cite{Yang:xx14}). Specifically,
Lemma~\ref{lem:per->bijection} and Proposition~\ref{prp:per group} follows from
\cite[Theorem~4.2(1)--(3)]{CarlsenKangEtAl:xxxx}; and Lemma~\ref{lem:unitary} and
Proposition~\ref{prp:unitary group} follow (with some effort) from
\cite[Theorem~4.2(5)~and~Proposition~3.3]{CarlsenKangEtAl:xxxx}. However, a simpler
direct argument works for strongly connected finite $k$-graphs, and we present that
instead.

To state our results we must briefly discuss infinite paths in $k$-graphs. The set
\[
    \Omega_k:=\{(m,n) \in \NN^k \times \NN^k : m \le n\}
\]
becomes a $k$-graph  with operations $r(m,n) = (m,m)$, $s(m,n) = (n,n)$, $(m,n)(n,p) =
(m,p)$ and $d(m,n) = n - m$. We identify $\Omega_k^0$ with $\NN^k$ via $(m,m) \mapsto m$.
An \emph{infinite path} in a $k$-graph $\Lambda$ is a functor $x : \Omega_k \to \Lambda$
that intertwines the degree maps. We write $\Lambda^\infty$ for the collection of all
infinite paths and call this the \emph{infinite-path space} of $\Lambda$. For $x \in
\Lambda^\infty$ we write $r(x)$ for $x(0)$. For $n \in \NN^k$, we write $\sigma^n(x)$ for
the infinite path such that $\sigma^n(x)(p,q) = x(n+p, n+q)$. If $r(x) = s(\lambda)$,
then there is a unique infinite path $\lambda x$ such that $(\lambda x)(0,d(\lambda)) =
\lambda$ and $\sigma^{d(\lambda)}(\lambda x) = x$. For $\lambda \in \Lambda$ we define
$Z(\lambda) = \{x \in \Lambda^\infty : x(0, d(\lambda)) = \lambda\}$. If $\Lambda$ has no
sources, then each $Z(\lambda)$ is nonempty.

We say $\Lambda$ is \emph{aperiodic} if for each $v \in \Lambda^0$, there exists $x \in
Z(v)$ such that for all $m\neq n\in\NN^k$ we have $\sigma^m(x) \not= \sigma^n(x)$. By
\cite[Lemma~3.2]{RobertsonSims:BLMS07}, $\Lambda$ is aperiodic if and only if there do
not exist $v \in \Lambda^0$ and $m \not= n \in \NN^k$ such that $\sigma^m(x) =
\sigma^n(x)$ for all $x \in Z(v)$.

\begin{lem}\label{lem:per->bijection}
Let $\Lambda$ be a strongly connected finite $k$-graph. Suppose that $v \in \Lambda^0$
amd $m,n \in \NN^k$ satisfy $\sigma^m(x) = \sigma^n(x)$ for all $x \in Z(v)$.
\begin{enumerate}
\item\label{it:lp->Per} For all $x \in \Lambda^\infty$ we have $\sigma^m(x) =
    \sigma^n(x)$.
\item\label{it:theta exists} For each $\mu \in \Lambda^m$ there exists a unique
    $\theta_{m,n}(\mu) \in \Lambda^n$ such that $\mu x = \theta_{m,n}(\mu) x$ for all
    $x \in Z(s(\mu))$. The map $\theta_{m,n} : \Lambda^m\to \Lambda^n$ is range- and
    source-preserving.
\item\label{it:theta composition} If $w \in \Lambda^0$ and $p \in \NN^k$ also satisfy
    $\sigma^n(x) = \sigma^p(x)$ for all $x \in Z(w)$, then $\sigma^m(x) =
    \sigma^p(x)$ for all $x \in \Lambda^\infty$, and
    $\theta_{n,p}\circ\theta_{m,n}=\theta_{m,p}$.
\item\label{it:theta prp} Each $\theta_{m,m} : \Lambda^m \to \Lambda^m$ is the
    identity map, and each $\theta_{m,n} : \Lambda^m \to \Lambda^n$ is a bijection
    with $\theta_{m,n}^{-1}=\theta_{n,m}$.
\end{enumerate}
\end{lem}
\begin{proof}
(\ref{it:lp->Per}) Fix $x \in \Lambda^\infty$. Since $\Lambda$ is strongly connected, $v
\Lambda r(x)$ has at least one element, say $\lambda$. So $\lambda x \in Z(v)$, and hence
\[
\sigma^m(x)
    = \sigma^{m + d(\lambda)}(\lambda x)
        = \sigma^{d(\lambda)}(\sigma^m(\lambda x))
        = \sigma^{d(\lambda)}(\sigma^n(\lambda x))
        = \sigma^{n + d(\lambda)}(\lambda x)
        = \sigma^n(x).
\]

(\ref{it:theta exists}) Fix $\mu \in \Lambda^m$. Since $\Lambda$ is strongly connected,
Lemma~\ref{lem-nosources}(\ref{lem-nosources-item2}) shows that there exists $\alpha \in
\Lambda^n r(\mu)$. Let $\beta := (\alpha\mu)(0, m)$ and let $\theta_{m,n}(\mu) :=
(\alpha\mu)(m, m+n)$. Fix $x \in Z(s(\mu))$. By (\ref{it:lp->Per}) applied to $\alpha\mu
x$,
\[
\mu x
    = \sigma^n(\alpha\mu x)
    = \sigma^m(\alpha\mu x)
    = \sigma^m(\beta \theta_{m,n}(\mu) x)
    = \theta_{m,n}(\mu) x.
\]
This implies in particular that $Z(\mu) = Z(\theta_{m,n}(\mu))$. Since the sets $Z(\nu)$
for $\nu \in \Lambda^n$ are mutually disjoint, $\theta_{m,n}(\mu)$ is the unique element
of $\Lambda^n$ such that $\mu x = \theta_{m,n}(\mu) x$.

This gives a function $\theta_{m,n}:\Lambda^m\to \Lambda^n$. We have
$s(\theta_{m,n}(\mu)) = s(\mu)$ and $r(\theta_{m,n}(\mu)) = r(\mu)$ by construction.

(\ref{it:theta composition}) Two applications of part~(\ref{it:lp->Per}) show that
$\sigma^m(x) = \sigma^n(x) = \sigma^p(x)$ for all $x \in \Lambda^\infty$. Let
$\mu\in\Lambda^m$ and $x\in Z(s(\mu))$. Then
\[
\theta_{n,p}(\theta_{m,n}(\mu))x=\theta_{m,n}(\mu)x=\mu x=\theta_{m,p}(\mu)x.
\]
Thus $\theta_{n,p}\circ\theta_{m,n}=\theta_{m,p}$ by the uniqueness assertion in
\eqref{it:theta exists}.

(\ref{it:theta prp}) Uniqueness in part~(\ref{it:theta exists}) shows that
$\theta_{m,m}(\mu) = \mu$ for all $\mu \in \Lambda^m$. Now $\theta_{n,m} \circ
\theta_{m,n} = \theta_{m,m} = \id_{\Lambda^m}$ by~(\ref{it:theta composition}), and
likewise $\theta_{m,n} \circ \theta_{n,m} = \id_{\Lambda^n}$.
\end{proof}

\begin{prop}\label{prp:per group}
Suppose that $\Lambda$ is a strongly connected finite $k$-graph.
\begin{enumerate}
\item\label{it:independent} Let $m,n,p,q\in\NN^k$. Suppose that $\sigma^m(x) =
    \sigma^n(x)$ for all $x \in \Lambda^\infty$. If $p - q = m - n$, then
    $\sigma^p(x) = \sigma^q(x)$ for all $x \in \Lambda^\infty$.
\item\label{it:group} The set
    \[
        \Per\Lambda := \{m - n : m,n \in \NN^k, \sigma^m(x) = \sigma^n(x)\text{ for all $x \in \Lambda^\infty$}\}
    \]
    is a subgroup of $\ZZ^k$.
\item\label{it:theta ext} Suppose that $m - n \in \Per\Lambda$ and that $\mu \in
    \Lambda^m$. Then $\theta_{d(\alpha) + m,\, d(\alpha) + n}(\alpha \mu) = \alpha
    \theta_{m,n}(\mu)$ and $\theta_{m + d(\beta),\, n + d(\beta)}(\mu\beta) =
    \theta_{m,n}(\mu)\beta$ for all $\alpha \in \Lambda r(\mu)$ and $\beta \in
    s(\mu)\Lambda$.
\end{enumerate}
\end{prop}
\begin{proof}
(\ref{it:independent}). Fix $x \in \Lambda^\infty$.
Lemma~\ref{lem-nosources}(\ref{lem-nosources-item2}) shows that there exists $\alpha \in
\Lambda^m r(x)$. We calculate:
\[
\sigma^p(x)
    = \sigma^{p+m}(\alpha x)
    = \sigma^m(\sigma^p(\alpha x))
    = \sigma^n(\sigma^p(\alpha x))
    = \sigma^{n+p}(\alpha x)
    = \sigma^{m+q}(\alpha x)
    = \sigma^q(x).
\]

(\ref{it:group}) We  have $0 \in \Per\Lambda$, and $-p \in \Per\Lambda$ whenever $p \in
\Per\Lambda$. If $m - n, p - q \in \Per\Lambda$, then for $x \in \Lambda^\infty$,
$\sigma^{p+m}(x) = \sigma^p(\sigma^m(x)) = \sigma^p(\sigma^n(x)) = \sigma^q(\sigma^n(x))
= \sigma^{q+n}(x)$. Thus $\Per\Lambda$ is closed under addition, and hence is a subgroup
of $\ZZ^k$.

(\ref{it:theta ext}) Let $\alpha \in \Lambda r(\mu)$, and fix $x \in Z(s(\mu))$. The
defining property of $\theta_{m,n}(\mu)$ implies that $\theta_{m,n}(\mu) x = \mu x$, and
hence $ \alpha\theta_{m,n}(\mu) x = \alpha\mu x$. Uniqueness in
Lemma~\ref{lem:per->bijection}(\ref{it:theta exists}) gives $\theta_{d(\alpha) + m,\,
d(\alpha) + n}(\alpha\mu) = \alpha \theta_{m,n}(\mu)$. A similar argument shows that
$\theta_{m + d(\beta),\, n + d(\beta)}(\mu\beta) = \theta_{m,n}(\mu)\beta$ for all $\beta
\in s(\mu)\Lambda$.
\end{proof}

\begin{cor}\label{cor:lmin<->theta}
Suppose that $\Lambda$ is a strongly connected finite $k$-graph. Suppose that $m - n \in
\Per\Lambda$ and $\mu \in \Lambda^m$. Let $p := (m \vee n) - m$ and $q := (m \vee n) -
n$. Then
\[
\Lambda^{\min}(\theta_{m,n}(\mu), \mu)
    = \{(\alpha, \theta_{q,p}(\alpha)) : \alpha \in s(\mu)\Lambda^q\}.
\]
\end{cor}
\begin{proof}
For the containment $\subseteq$, suppose that $(\alpha,\beta) \in
\Lambda^{\min}(\theta_{m,n}(\mu), \mu)$. Then $d(\alpha) = q$ and $d(\beta) = p$ by
definition, and Lemma~\ref{lem:per->bijection}(\ref{it:theta exists}) gives $r(\alpha) =
s(\theta_{m,n}(\mu)) = s(\mu)$. For $x \in Z(s(\alpha))$ we have
\[
\alpha x = \sigma^n(\theta_{m,n}(\mu)\alpha x)
    = \sigma^n(\mu\beta x)
    = \sigma^m(\mu\beta x)=\beta x
\]
because $m - n \in \Per\Lambda$. Thus $\beta = \theta_{q,p}(\alpha)$.

For the containment $\supseteq$, fix $\alpha \in s(\mu)\Lambda^q$. Let $x \in
Z(s(\alpha))$. We have
\[
\theta_{m,n}(\mu)\alpha x
    = \mu\alpha x
    = \mu\theta_{q,p}(\alpha) x.
\]
The factorisation property implies that $\theta_{m,n}(\mu)\alpha =
\mu\theta_{q,p}(\alpha)$. Since $n+q=m+p=m\vee n$ we have $(\alpha, \theta_{q,p}(\alpha))
\in \Lambda^{\min}(\theta_{m,n}(\mu), \mu)$.
\end{proof}

\begin{prop}\label{prp:Per<->per}
Suppose that $\Lambda$ is a strongly connected finite $k$-graph. Then $\Lambda$ is
aperiodic if and only if $\Per\Lambda = \{0\}$.
\end{prop}
\begin{proof}
First suppose that $\Lambda$ is aperiodic, and take $m - n \in \Per\Lambda$.
Proposition~\ref{prp:per group}(\ref{it:independent}) implies that $\sigma^m(x) =
\sigma^n(x)$ for all $x \in \Lambda^\infty$. Since $\Lambda$ is aperiodic, this forces $m
= n$. Hence $\Per\Lambda = \{0\}$.

Now suppose that $\Lambda$ is not aperiodic. The equivalence of (i)~and~(iii) in
\cite[Lemma~3.2]{RobertsonSims:BLMS07} implies that there exist $m \not= n \in \NN^k$ and
$v \in \Lambda^0$ such that $\sigma^m(x) = \sigma^n(x)$ for all $x \in Z(v)$.
Lemma~\ref{lem:per->bijection}(\ref{it:lp->Per}) then implies that $\sigma^m(x) =
\sigma^n(x)$ for all $x \in \Lambda^\infty$, and hence $m - n \in \Per\Lambda \setminus
\{0\}$.
\end{proof}

\begin{ex}
Suppose that $\Lambda$ is a finite $2$-graph with one vertex. This puts us in the
situation studied by Davidson and Yang in \cite{DY}. The group $\Per\Lambda$ is then the
intersection, over all infinite paths $x$ in $\Lambda$, of the associated symmetry groups
$H_x$ discussed in \cite[Section~2]{DY}. Proposition~\ref{prp:per
group}(\ref{it:independent}) boils down to equivalence of (i)~and~(ii) in
\cite[Theorem~3.1]{DY}. The bijections $\theta_{m,n}$ of
Proposition~\ref{lem:per->bijection}(\ref{it:theta prp}) are the bijections $\gamma$ of
\cite[Theorem~3.1(iii)]{DY}.
\end{ex}

\section{A central representation of the periodicity group}\label{sec:central rep}

We now describe how the group $\Per\Lambda$ shows up in $C^*(\Lambda)$.

\begin{prop}\label{prp:unitary group}
Let $\Lambda$ be a strongly connected finite $k$-graph, and for $m,n \in \NN^k$ such that
$m-n \in \Per\Lambda$, let $\theta_{m,n}$ be the bijection of
Lemma~\ref{lem:per->bijection}. There is a unitary representation $U$ of $\Per\Lambda$ in
the centre of $C^*(\Lambda)$ such that $U_{m - n} = \sum_{\mu \in \Lambda^m} s_{\mu}
s^*_{\theta_{m,n}(\mu)}$ whenever $m - n \in \Per\Lambda$.
\end{prop}

\begin{lem}\label{lem:unitary}
Let $\Lambda$ be a strongly connected finite $k$-graph. Suppose that $m - n \in
\Per\Lambda$. Then $s_\mu s^*_\mu = s_{\theta_{m,n}(\mu)} s^*_{\theta_{m,n}(\mu)}$ for
all $\mu \in \Lambda^m$. The element $U := \sum_{\mu \in \Lambda^m} s_{\mu}
s^*_{\theta_{m,n}(\mu)}$ is a unitary in $C^*(\Lambda)$.
\end{lem}

\begin{proof}
Let $p = (m \vee n) - m$, $q=(m\vee n)-n$ and  $\mu \in \Lambda^m$. By
Corollary~\ref{cor:lmin<->theta},
\[
\Lambda^{\min}(\mu,\theta_{m,n}(\mu))=\{(\theta_{q,p}(\alpha),\alpha):\alpha\in s(\mu)\Lambda^q\},
\]
and in particular, $\mu\theta_{q,p}(\alpha) = \theta_{m,n}(\mu)\alpha$ for all $\alpha
\in s(\mu)\Lambda^q$. Using this at the fourth equality, we compute:
\begin{align*}
s_\mu s^*_\mu
    &= s_\mu \Big(\sum_{\beta \in s(\mu)\Lambda^p} s_\beta s^*_\beta\Big) s^*_\mu
    = s_\mu\Big(\sum_{\alpha\in s(\mu)\Lambda^q} s_{\theta_{q,p}(\alpha)}s_{\theta_{q,p}(\alpha)}^* \Big)s_\mu^*\\
    &= \sum_{\alpha\in s(\mu)\Lambda^q} s_{\mu \theta_{q,p}(\alpha)} s_{\mu\theta_{q,p}(\alpha)}^*
    = \sum_{\alpha\in s(\mu)\Lambda^q} s_{\theta_{m,n}(\mu)\alpha} s_{\theta_{m,n}(\mu)\alpha}^*\\
    &= s_{\theta_{m,n}(\mu)}\Big(\sum_{\alpha\in s(\mu)\Lambda^q} s_{\alpha} s_{\alpha}^*\Big)s^*_{\theta_{m,n}(\mu)}
    = s_{\theta_{m,n}(\mu)} s^*_{\theta_{m,n}(\mu)}.
\end{align*}
Since $\theta_{m,n}$ is a source-preserving bijection we have
\[
UU^* = \sum_{\mu,\eta \in \Lambda^m} s_\mu s^*_{\theta_{m,n}(\mu)} s_{\theta_{m,n}(\eta)} s^*_\eta
     = \sum_{\mu \in \Lambda^m} s_\mu s^*_{\theta_{m,n}(\mu)} s_{\theta_{m,n}(\mu)} s_\mu^*
     =\sum_{v\in \Lambda^0}\sum_{\mu\in v\Lambda^m}s_\mu s^*_\mu
     = 1_{C^*(\Lambda)}.
\]
The symmetric calculation gives $UU^* = 1$. Thus $U$ is unitary.
\end{proof}

\begin{proof}[Proof of Proposition~\ref{prp:unitary group}]
We start by showing that $\sum_{\mu \in \Lambda^m} s_{\mu} s^*_{\theta_{m,n}(\mu)}$
depends only on $m-n$. Suppose that $m,n,p,q\in\NN^k$ and $m - n = p - q\in \Per\Lambda$.
Then~(CK4) followed by Proposition~\ref{prp:per group}(\ref{it:theta ext}) imply that
\[
\sum_{\mu \in \Lambda^m} s_{\mu} s^*_{\theta_{m,n}(\mu)}
    = \sum_{\mu \in \Lambda^m}\sum_{\alpha \in s(\mu)\Lambda^p}s_{\mu} s_\alpha s^*_\alpha s^*_{\theta_{m,n}(\mu)}
    = \sum_{\eta \in \Lambda^{m + p}}s_{\eta} s^*_{\theta_{m + p, n + p}(\eta)}.
\]
The same calculation with $(m,n,p)$ replaced by $(p,q,m)$ gives
\[
\sum_{\nu \in \Lambda^p} s_{\nu} s^*_{\theta_{p,q}(\nu)}
    = \sum_{\zeta \in \Lambda^{p + m}} s_{\zeta} s^*_{\theta_{p + m, q + m}(\zeta)}.
\]
Since $n + p = q + m$, the formula for $U_{m-n}$ is well defined.

Lemma~\ref{lem:unitary} implies that $U_{m-n}$ is unitary.  By
Lemma~\ref{lem:per->bijection}(\ref{it:theta exists}), $\theta_{n,m} =\theta_{m,n}^{-1}$,
and hence $U_{m - n} = U^*_{n - m}$. To see that $g \mapsto U_g$ is a homomorphism, fix
$g,h \in \Per\Lambda$. To line things up, choose $g_{+}, g_{-}, h_{+}, h_{-}\in\NN^k$
such that $g = g_+ - g_-$ and $h = h_+ -h_-$. Let $m := g_+ + h_+$, $n := g_- + h_+$ and
$p := g_- + h_-$. Then $g = m - n$ and $h = n - p$. For $\mu \in \Lambda^m$ and $\nu \in
\Lambda^n$, we have $s^*_{\theta_{m,n}(\mu)} s_\nu = \delta_{\theta_{m,n}(\mu), \nu}
p_{s(\mu)}$ by (CK3)~and~(CK4). This and Lemma~\ref{lem:per->bijection}(\ref{it:theta
composition}) give
\[
U_g U_h
    = \sum_{\mu \in \Lambda^m}\sum_{\nu \in \Lambda^n}
            s_\mu s^*_{\theta_{m,n}(\mu)} s_{\nu} s^*_{\theta_{n,p}(\nu)}
    =\sum_{\mu \in \Lambda^m} s_\mu s^*_{\theta_{n,p}(\theta_{m,n}(\mu))}
    =\sum_{\mu \in \Lambda^m} s_\mu s^*_{\theta_{m,p}(\mu)}.
\]
Since $m - p = m - n + n - p = g + h$, we deduce that $U_g U_h = U_{g + h}$.

To see that the $U_g$ are central, it suffices to show that $U_g s_\lambda = s_\lambda
U_g$ for all $g\in \Per\Lambda$ and $\lambda\in\Lambda$: since $\Per\Lambda$ is a group
and $U_{-g} = U_g^*$ we then have $U_g s^*_\lambda = (s_\lambda U_{-g})^* = (U_{-g}
s_\lambda)^* = s^*_\lambda U_g$. Fix $\lambda \in \Lambda$ and $g \in \Per\Lambda$.
Choose $m,n \in \NN^k$ such that $g = m-n$, and let $p := m + d(\lambda)$ and $q := n +
d(\lambda)$. By factoring $\xi\in\Lambda^p$ into paths of degree $d(\lambda)$ and $m$,
\[
U_g s_\lambda
    = \sum_{\xi \in \Lambda^p} s_\xi s^*_{\theta_{p,q}(\xi)} s_\lambda
    = \sum_{\eta \in \Lambda^{d(\lambda)}} \sum_{\mu \in s(\eta)\Lambda^m} s_{\eta\mu} s^*_{\theta_{p,q}(\eta\mu)} s_\lambda.
\]
By Proposition~\ref{prp:per group}(\ref{it:theta ext}),  each $\theta_{p,q}(\eta\mu) =
\eta \theta_{m,n}(\mu)$. Since  $s^*_{\eta \theta_{m,n}(\mu)} s_\lambda =
\delta_{\eta,\lambda} s^*_{\theta_{m,n}(\mu)}$ we deduce that
\[
U_g s_\lambda = \sum_{\mu \in
s(\lambda)\Lambda^m} s_{\lambda\mu} s^*_{\theta_{m,n}(\mu)}
    = \sum_{\mu \in \Lambda^m} s_{\lambda} s_{\mu} s^*_{\theta_{m,n}(\mu)}
    = s_\lambda U_g.\qedhere
\]
\end{proof}

\section{The statement of the main result}\label{sec:main}

Suppose that $\Lambda$ is a strongly connected finite $k$-graph. Our main theorem,
Theorem~\ref{thm:main} below, describes the KMS$_1$ states of $C^*(\Lambda)$ for the
\emph{preferred dynamics} defined by
\[
\alpha_t = \gamma_{\rho(\Lambda)^{it}}\quad \text{ for all $t \in \RR$}
\]
corresponding to $r=\ln\rho(\Lambda)$.

To see why we chose this dynamics and inverse temperature, take $r \in \RR^k$ and $\beta
\in [0,\infty)$ and let $\alpha^r$ be the dynamics $\alpha^r_t = \gamma_{e^{itr}}$.
Suppose that $\phi$ is a KMS$_\beta$ state for $(C^*(\Lambda), \alpha^r)$. Then
Corollary~\ref{cor:aHLRS2-Cor4.3,4.4}(\ref{it:KMSCK}) implies that $\beta r =
\ln\rho(\Lambda)$. So $\alpha_t = \alpha^r_{\beta t}$ for all $t$, and hence the
KMS$_\beta$ condition for $\alpha^r$ is the KMS$_1$ condition for $\alpha$. So $\phi$ is
a KMS$_1$ state for $(C^*(\Lambda), \alpha)$.

There is a slight subtlety here when $\rho(\Lambda) = (1, \dots, 1)$. The preferred
dynamics is then the trivial action, and so the KMS$_1$ states described in
Theorem~\ref{thm:main} are traces, and are KMS$_\beta$ states for all other values of
$\beta$. If at least one $\rho(\Lambda)_i$ is different from 1, then
Corollary~\ref{cor:aHLRS2-Cor4.3,4.4}(\ref{it:KMSCK}) shows that $(C^*(\Lambda), \alpha)$
admits KMS$_\beta$ states only for $\beta = 1$.

\begin{thm}\label{thm:main}
Suppose that $\Lambda$ is a strongly connected finite $k$-graph. Let $\alpha$ be the
preferred dynamics on $C^*(\Lambda)$. Let $\pi_U : C^*(\Per\Lambda) \to C^*(\Lambda)$ be
the homomorphism of Proposition~\ref{prp:unitary group}. Then $\pi_U^* : \phi \mapsto
\phi \circ \pi_U$ is an affine isomorphism of the KMS$_1$ simplex of $(C^*(\Lambda),
\alpha)$ onto the state space of $C^*(\Per\Lambda)$.
\end{thm}

The proof of Theorem~\ref{thm:main} occupies the next three sections. The proof strategy
is as follows. In Section~\ref{sec:measures}, we show that the KMS states of
$(C^*(\Lambda), \alpha)$ all induce the same measure $M$ on the spectrum of the abelian
subalgebra $\clsp\{s_\lambda s^*_\lambda\} \subseteq C^*(\Lambda)$, and we characterise
$M$ in terms of the unimodular Perron-Frobenius eigenvector $x^\Lambda$. We use $M$ in
Section~\ref{sec:injectivity} to establish a formula for a KMS state $\phi$ in terms of
$\phi \circ \pi_U$ (see Theorem~\ref{thm:CK states}). In Section~\ref{sec:surjectivity}
we use $M$ again to construct a particular KMS state in Proposition~\ref{prp:integral
KMS}. This state is not always supported on $\clsp\{s_\lambda s^*_\lambda\}$, so
composing with the gauge automorphisms yields more KMS$_1$ states
(Corollary~\ref{cor:many KMS}). We can then prove Theorem~\ref{thm:main}: we deduce from
the formula for KMS states established in Theorem~\ref{thm:CK states} that $\pi^*_U$ is a
continuous affine injection; we then use Corollary~\ref{cor:many KMS} to see that each
pure state of $C^*(\Per\Lambda)$ is in the image of $\pi^*_U$, and deduce that $\pi^*_U$
is surjective.

Before moving on to the first part of the proof of Theorem~\ref{thm:main}, a reality
check is in order. If $\Lambda$ is coordinatewise irreducible, in the sense that each
$A_i$ is an irreducible matrix, then it is also strongly connected. So both
Theorem~\ref{thm:main} and \cite[Theorem~7.2]{aHLRSk-graphs} apply. The next remark
reconciles the hypotheses of the two results.

\begin{rmk}
Suppose that $\Lambda$  is coordinatewise-irreducible. Theorem~\ref{thm:main} says that
if $\Per\Lambda$ is nontrivial, then $(C^*(\Lambda), \alpha)$ has many KMS states.
Theorem~7.2 of \cite{aHLRSk-graphs}, on the other hand, says that if the coordinates of
the vector $\ln\rho(\Lambda)$ are rationally independent, then $(C^*(\Lambda), \alpha)$
admits a unique KMS state. To reconcile the two results, we will show that if
$\Per\Lambda$ is nontrivial, then the coordinates of $\ln\rho(\Lambda)$ are rationally
dependent.

Let  $m - n \in \Per\Lambda \setminus \{0\}$. By
Lemma~\ref{lem:per->bijection}(\ref{it:theta prp}), there is a source- and
range-preserving bijection of $\Lambda^m$ onto $\Lambda^n$.  For $v, w\in\Lambda^0$, we
have $A^m(v,w) = |v\Lambda^m w| = |v\Lambda^n w| = A^n(v,w)$, and so $A^m=A^n$. Since
$\Lambda$ is strongly connected, Corollary~\ref{cor-PFgraph}(\ref{cor-PFgraph0}) shows
that
\begin{equation}\label{eq:rationally dependent}
\rho(A)^m
    = \prod_i\rho(A_i)^{m_i}
    = \rho(A^m)
    = \rho(A^n)
    = \prod_i\rho(A_i)^{n_i}
    = \rho(A)^n.
\end{equation}
Taking logarithms,
\[
m \cdot \ln\rho(\Lambda)=\sum_{i=1}^k m_i\ln\rho(A_i)
    = \ln\Big(\prod_{i=1}^k \rho(A_i)^{m_i}\Big)
    = \ln\Big(\prod_{i=1}^k \rho(A_i)^{n_i}\Big)
    = n \cdot \ln\rho(\Lambda).
\]
Thus the coordinates of $\ln\rho(\Lambda)$ are rationally dependent.
\end{rmk}

\section{Measures on the infinite-path space}\label{sec:measures}

Let $\Lambda$ be a strongly connected finite $k$-graph. Each KMS state of $C^*(\Lambda)$
restricts to a state of the commutative subalgebra $\clsp\{s_\lambda s^*_\lambda :
\lambda \in \Lambda\}$, and hence to a probability measure on its spectrum satisfying an
invariance condition (see~\eqref{eq:measure cond}). In this section we use our
Perron-Frobenius theorem and results of Choksi \cite{Choksi:PLMS58} about measures on
inverse-limit spaces to show that there is a unique measure satisfying~\eqref{eq:measure
cond}. We will use this result in Section~\ref{sec:injectivity} to give a formula for a
KMS state $\phi$ in terms of its restriction to the image of $C^*(\Per\Lambda)$, and
again in Section~\ref{sec:surjectivity} to construct KMS states.

Recall from \cite{KP} that the sets $Z(\lambda) = \{x \in \Lambda^\infty :
x(0,d(\lambda)) = \lambda\}$ indexed by $\lambda \in \Lambda$ constitute a basis of
compact open sets for a compact Hausdorff topology on $\Lambda^\infty$. Equip the finite
sets $\Lambda^m$ with the discrete topology. Let $m \le n\in\NN^k$ and define $\pi_{m,n}
: \Lambda^n \to \Lambda^m$ by $\pi_{m,n}(\lambda)=\lambda(0,m)$. Then $(\Lambda^m,
\pi_{m,n})$ is an inverse system of compact topological spaces and continuous, surjective
maps. Using the universal property of the inverse limit it is routine to show that $x
\mapsto \big(x(0,m)\big)_{m \in \NN^k}$ is a homeomorphism of $\Lambda^\infty$ onto the
inverse limit $\varprojlim(\Lambda^m, \pi_{m,n})$.

There is an isomorphism of the commutative subalgebra $\clsp\{s_\lambda s^*_\lambda :
\lambda \in \Lambda\}$ of $C^*(\Lambda)$ onto $C(\Lambda^\infty)$ that carries $s_\lambda
s^*_\lambda$ to $1_{Z(\lambda)}$ (see,  for example,  Theorem~7.1 of
\cite{Webster:SM11}). Thus the Riesz Representation Theorem associates to each state
$\phi$ of $C^*(\Lambda)$ a Borel probability measure $M$ on $\Lambda^\infty$ such that
$M(Z(\lambda)) = \phi(s_\lambda s^*_\lambda)$.

Let $\alpha$ denote the preferred dynamics on $C^*(\Lambda)$, and suppose that $\phi$ is
a KMS$_1$ state of $(C^*(\Lambda), \alpha)$. The KMS condition ensures that, for
$\lambda\in\Lambda$,
\begin{equation}
\phi(s_\lambda s^*_\lambda) = \rho(\Lambda)^{-d(\lambda)} \phi(s^*_\lambda s_\lambda)
    = \rho(\Lambda)^{-d(\lambda)} \phi(p_{s(\lambda)}),
\end{equation}
and hence the corresponding probability measure $M$ on $\Lambda^\infty$ satisfies
\begin{equation}\label{eq:measure cond}
    M(Z(\lambda)) = \rho(\Lambda)^{-d(\lambda)} M(Z(s(\lambda)))\quad\text{ for all $\lambda\in\Lambda$}.
\end{equation}
We now show that there is exactly one measure satisfying \eqref{eq:measure cond}.

\begin{prop}\label{prop:measure e-u}
Suppose that $\Lambda$ is a strongly connected finite $k$-graph. Then there exists a
unique Borel probability measure $M$ on $\Lambda^\infty$ that satisfies~\eqref{eq:measure
cond}. Let $x^\Lambda$ be the unimodular Perron-Frobenius eigenvector of $\Lambda$. We
have
\begin{equation}\label{eq:cylinder measure}
    M(Z(\lambda)) = \rho(\Lambda)^{-d(\lambda)} x^\Lambda_{s(\lambda)}\quad\text{ for all $\lambda$.}
\end{equation}
\end{prop}

\begin{proof}
We build a measure $M$ satisfying~\eqref{eq:measure cond} and~\eqref{eq:cylinder measure}
by viewing  $\Lambda^\infty$ as the inverse limit of the sets $\Lambda^m$ under the maps
$\pi_{m,n} : \Lambda^n \to \Lambda^m$ for  $n\geq m\in\NN^k$. For $S\subseteq \Lambda^m$,
define $M_m(S) = \rho(\Lambda)^{-m}\sum_{\mu \in S}
 x^\Lambda_{s(\mu)}$. Then $M_m$ is a  measure on $\Lambda^m$.

For $m \le n$ and $\mu \in \Lambda^m$, we have
\begin{align*}
M_n(\pi_{m,n}^{-1}(\{\mu\}))
    &= \sum_{\mu' \in s(\mu)\Lambda^{n-m}} \rho(\Lambda)^{-n} x^\Lambda_{s(\mu')}
    = \rho(\Lambda)^{-n} \sum_{w \in \Lambda^0} A^{n-m}(s(\mu),w) x^\Lambda_w\\
    &= \rho(\Lambda)^{-n} (A^{n-m} x^\Lambda)_{s(\mu)}= \rho(\Lambda)^{-m} x^\Lambda_{s(\mu)}
    = M_m(\{\mu\}).
\end{align*}
Thus $M_n(\pi_{m,n}^{-1}(S)) = M_m(S)$ for all $S \subseteq \Lambda^m$ and the measure
spaces $((\Lambda^m, M_m), \pi_{m,n})$ form an inverse system. Theorem~2.2 of
\cite{Choksi:PLMS58} implies that there is a Borel measure $M$ on $\Lambda^\infty =
\varprojlim (\Lambda^m, \pi_{m,n})$ such that, for $\mu \in \Lambda^m$,
\[
M(Z(\mu))
    = M_m(\{\mu\})
    = \rho(\Lambda)^{-m} x^\Lambda_{s(\mu)}
    = \rho(\Lambda)^{-m}M(Z(s(\mu))).
\]
Since $M(\Lambda^\infty)=\sum_{v\in\Lambda^0} M(Z(v))=\sum_{v\in\Lambda^0}
x_v^\Lambda=1$, this $M$ is a probability measure satisfying~\eqref{eq:measure cond}
and~\eqref{eq:cylinder measure}.

Now suppose that $M'$ is a Borel probability measure satisfying~\eqref{eq:measure cond}.
Define a vector $y \in [0,\infty)^{\Lambda^0}$ by $y_v = M'(Z(v))$ for each $v$. For
$1\le i \le k$, we have $Z(v) = \bigsqcup_{\alpha \in v\Lambda^{e_i}} Z(\alpha)$, and
using~\eqref{eq:measure cond}, we have
\begin{align*}
\rho(A_i) y_v
    &= \rho(A_i) M'(Z(v))
    = \rho(A_i) \sum_{\alpha \in v\Lambda^{e_i}} M'(Z(\alpha)) \\
    &= \sum_{\alpha \in v\Lambda^{e_i}} M'(Z(s(\alpha)))
    = \sum_{w \in \Lambda^0} A_i(v,w) y_w
    = (A_i y)_v.
\end{align*}
So $y$ is a non-negative eigenvector of each $A_i$ with eigenvalue $\rho(A_i)$ and unit
$1$-norm. Thus $y = x^\Lambda$ by Corollary~\ref{cor-PFgraph}(\ref{cor-PFgraph1}).
Now~\eqref{eq:cylinder measure} for $M'$ follows from~\eqref{eq:measure cond}. Thus
$M=M'$.
\end{proof}

Each $\sigma^m : \Lambda^\infty \to \Lambda^\infty$ is continuous since it restricts to a
homeomorphism of $Z(\mu)$ for each $\mu \in \Lambda^m$. So each $\{x \in \Lambda^\infty :
\sigma^m(x) = \sigma^n(x)\}$ is closed and hence Borel. We next show that when $m-n \in
\Per\Lambda$, the measure $M$ is supported on $\{x \in \Lambda^\infty : \sigma^m(x) =
\sigma^n(x)\}$.

\begin{prop}\label{prp:m sees Per}
Let $\Lambda$ be a strongly connected finite $k$-graph, and let $M$ be the measure on
$\Lambda^\infty$ obtained from Proposition~\ref{prop:measure e-u}. For $m,n \in \NN^k$,
we have
\[
M\big(\{x \in \Lambda^\infty : \sigma^m(x) = \sigma^n(x)\}\big)
    = \begin{cases}
        1 &\text{ if $m-n \in \Per\Lambda$}\\
        0 &\text{ otherwise.}
    \end{cases}
\]
\end{prop}

The proof of the proposition requires the following two technical lemmas.

\begin{lem}\label{lem:separating tails}
Let $\Lambda$ be a strongly connected finite $k$-graph. Suppose $g \in \ZZ^k \setminus
\Per\Lambda$. Then there exist $a \in \NN^k\setminus\{0\}$ and, for each $v\in
\Lambda^0$, a path $\lambda_v \in v\Lambda^a$ such that for $\mu,\nu\in\Lambda$ with
$s(\mu) = s(\nu)$ and $d(\mu) - d(\nu) = g$ we have $\Lambda^{\min}(\mu\lambda_{s(\mu)},
\nu\lambda_{s(\mu)}) = \emptyset$.
\end{lem}
\begin{proof}
Let $m := g \vee 0$ and $n := -g \vee 0$. Then $g = m - n$ and whenever $m', n' \in
\NN^k$ satisfy $m' - n' = g$, we have $m' \ge m$ and $n' \ge n$.

Since $g \not\in \Per\Lambda$, there exists $x \in \Lambda^\infty$ such that $\sigma^m(x)
\not= \sigma^n(x)$. So there exists $l \in \NN^k\setminus\{0\}$ such that $\sigma^m(x)(0,
l) \not= \sigma^n(x)(0, l)$. For each $v \in \Lambda^0$ there exists $\tau_v \in v\Lambda
r(x)$ because $\Lambda$ is strongly connected. Let $a := m + n + l + \bigvee_{v \in
\Lambda^0} d(\tau_v)$. For each $v \in \Lambda^0$ define
\[
\lambda_v := \tau_v x(0, a - d(\tau_v)).
\]

Fix $\mu, \nu \in \Lambda$ such that $d(\mu) - d(\nu) = g$ and $s(\mu) = s(\nu) = v$.
Then $d(\mu)\geq m$, $d(\nu)\geq n$, and there exists $p \in \NN^k$ such that $d(\mu) = m
+ p$ and $d(\nu) = n + p$. Factorise $\mu = \alpha\mu'$ and $\nu = \beta\nu'$ where
$d(\alpha) = d(\beta) = p$, so that $d(\mu') = m$ and $d(\nu') = n$. If $\alpha \not=
\beta$, then $\Lambda^{\min}(\mu,\nu) = \emptyset$ and hence
$\Lambda^{\min}(\mu\lambda_v,\nu\lambda_v) = \emptyset$. So we suppose that $\alpha =
\beta$. Then $\Lambda^{\min}(\mu \lambda_v, \nu \lambda_v) = \Lambda^{\min}(\mu'
\lambda_v, \nu'\lambda_v)$. We have
\begin{align*}
(\mu'\lambda_v)(m + n + d(\tau_v), m + n + d(\tau_v) + l)
    &= \lambda_v(n + d(\tau_v), n + d(\tau_v) + l)\\
    &=x(n, n+l)= \sigma^n(x)(0, l).
\end{align*}
Similarly $(\nu'\lambda_v)(m + n + d(\tau_v), m + n + d(\tau_v) + l) = \sigma^m(x)(0,
l)$. Since $\sigma^m(x)(0, l) \not= \sigma^n(x)(0, l)$ by choice of $x$ and $l$, the
factorisation property gives $\Lambda^{\min}(\mu\lambda_v, \nu\lambda_v) = \emptyset$.
\end{proof}

\begin{lem}\label{lem:measure estimate}
Let $\Lambda$ be a strongly connected finite $k$-graph, and let $M$ be the measure on
$\Lambda^\infty$ obtained in Proposition~\ref{prop:measure e-u}. Suppose that $g \in
\ZZ^k \setminus \Per\Lambda$. There exist $a \in \NN^k \setminus\{0\}$ and $0 < K  < 1$
such that whenever $s(\mu) = s(\nu)$ and $d(\mu) - d(\nu) = g$, we have
\begin{equation}\label{eq:measure estimate}
    M\bigg(\bigcup_{\substack{\lambda \in s(\mu)\Lambda^{ja}\\ \Lambda^{\min}(\mu\lambda,\nu\lambda) \not= \emptyset}}
        Z(\mu\lambda)\bigg) \le K^j M(Z(\mu))\quad\text{ for all $j \in \NN$.}
\end{equation}
\end{lem}
\begin{proof}
By Lemma~\ref{lem:separating tails} there exist $a \in \NN^k\setminus\{0\}$ and
$\lambda_v \in v\Lambda^a$ for each $v \in \Lambda^0$ such that
$\Lambda^{\min}(\mu\lambda_v, \nu\lambda_v) = \emptyset$ whenever $\mu,\nu \in \Lambda v$
satisfy $d(\mu) - d(\nu) = g$.

Let $v\in\Lambda^0$. Equation~\eqref{eq:cylinder measure} implies that  $0 <
M(Z(\lambda_v))$. Thus $M(Z(v) \setminus Z(\lambda_v)) < M(Z(v))$. Since $\Lambda^0$ is
finite, there exists  $0 < K < 1$ such that
\[
    M(Z(v) \setminus Z(\lambda_v)) < K M(Z(v)) < M(Z(v))\quad\text{ for all $v \in \Lambda^0$.}
\]
Fix $\mu,\nu$ such that $s(\mu) = s(\nu)$ and $d(\mu) - d(\nu) = g$. We
prove~\eqref{eq:measure estimate} by induction on $j$. When $j = 0$ both sides
of~\eqref{eq:measure estimate} are just $M(Z(\mu))$, so the inequality is trivial.

Now suppose that~\eqref{eq:measure estimate} holds for some $j \ge 0$. If $\eta,\zeta \in
\Lambda$ satisfy $\Lambda^{\min}(\eta, \zeta) = \emptyset$, then
$\Lambda^{\min}(\eta\xi,\zeta\xi) = \emptyset$ for all $\xi$. Using this for the second
equality, we calculate:
\begin{align*}
\bigcup_{\substack{\lambda \in s(\mu)\Lambda^{(j+1)a}\\ \Lambda^{\min}(\mu\lambda,\nu\lambda) \not= \emptyset}}
        Z(\mu\lambda)
    &=\bigcup_{\eta \in s(\mu)\Lambda^{ja}}
       \bigcup_{\substack{\xi \in s(\eta)\Lambda^a\\ \Lambda^{\min}(\mu\eta\xi, \nu\eta\xi) \not= \emptyset}}
            Z(\mu\eta\xi) \\
    &= \bigcup_{\substack{\eta \in s(\mu)\Lambda^{ja}\\ \Lambda^{\min}(\mu\eta,\nu\eta) \not= \emptyset}}
      \bigcup_{\substack{\xi \in s(\eta)\Lambda^a\\ \Lambda^{\min}(\mu\eta\xi, \nu\eta\xi) \not= \emptyset}}
        Z(\mu\eta\xi)
    \subseteq \bigcup_{\substack{\eta \in s(\mu)\Lambda^{ja}\\ \Lambda^{\min}(\mu\eta,\nu\eta) \not= \emptyset}}
      \bigcup_{\xi \in s(\eta)\Lambda^a \setminus \{\lambda_{s(\eta)}\}}
        Z(\mu\eta\xi).
\end{align*}
Hence
\begin{flalign*}
&&M\bigg(\bigcup_{\substack{\lambda \in s(\mu)\Lambda^{(j+1)a}\\
                    \Lambda^{\min}(\mu\lambda,\nu\lambda) \not= \emptyset}}
        Z(\mu\lambda)\bigg)
    &\le \sum_{\substack{\eta \in s(\mu)\Lambda^{ja}\\ \Lambda^{\min}(\mu\eta,\nu\eta) \not= \emptyset}}
            \rho(\Lambda)^{-d(\mu\eta)} \sum_{\xi \in s(\eta)\Lambda^a \setminus \{\lambda_{s(\eta)}\}}
                M(Z(\xi)) &\\
    &&&= \sum_{\substack{\eta \in s(\mu)\Lambda^{ja}\\ \Lambda^{\min}(\mu\eta,\nu\eta) \not= \emptyset}}
            \rho(\Lambda)^{-d(\mu\eta)} M\big(Z(s(\eta)) \setminus Z(\lambda_{s(\eta)})\big) &\\
    &&&< K \sum_{\substack{\eta \in s(\mu)\Lambda^{ja}\\ \Lambda^{\min}(\mu\eta,\nu\eta) \not= \emptyset}}
            \rho(\Lambda)^{-d(\mu\eta)} M(Z(s(\eta))) &\\
    &&&= K M\bigg(\bigcup_{\substack{\eta \in s(\mu)\Lambda^{ja}\\ \Lambda^{\min}(\mu\eta,\nu\eta) \not= \emptyset}}
            Z(\mu\eta)\bigg) &\\
    &&&\le K^{j+1} M(Z(\mu)) &
\end{flalign*}
by the induction hypothesis.
\end{proof}

\begin{proof}[Proof of Proposition~\ref{prp:m sees Per}]
Let $m-n \in \Per\Lambda$. Then $M\big(\{x \in \Lambda^\infty : \sigma^m(x) =
\sigma^n(x)\}\big) = M(\Lambda^\infty) = 1$ because $M$ is a probability measure.

Now suppose that $m - n \not\in \Per\Lambda$. Let $a$ and $K$ be as in
Lemma~\ref{lem:measure estimate}. Fix $j \in \NN$. We claim that
\[
\{x \in \Lambda^\infty : \sigma^m(x) = \sigma^n(x)\}
    \subseteq \bigcup_{\mu \in \Lambda^m, \nu\in \Lambda^n s(\mu)}
    \bigcup_{\substack{\lambda \in s(\mu)\Lambda^{ja}\\ \Lambda^{\min}(\mu\lambda,\nu\lambda) \not= \emptyset}}
        Z(\mu\lambda).
\]
To see this, let $x\in \Lambda^\infty$ and suppose that $\sigma^m(x) = \sigma^n(x)$. Let
$\mu := x(0,m)$, $\nu := x(0,n)$ and $\lambda := \sigma^m(x)(0, ja) = \sigma^n(x)(0,
ja)$. Then $x(0, (m \vee n) + ja) = \mu\lambda\alpha = \nu\lambda\beta$ for some
$\alpha,\beta$, and then $(\alpha,\beta) \in \Lambda^{\min}(\mu\lambda,\nu\lambda)$.
Since $x \in Z(\mu\lambda)$, this establishes the claim. Now Lemma~\ref{lem:measure
estimate} implies that for all $j$,
\[
M(\{x \in
\Lambda^\infty : \sigma^m(x)= \sigma^n(x)\}) \leq \sum_{\mu\in\Lambda^m, \nu\in\Lambda^n} K^jM(Z(\mu)) \le |\Lambda^m|\cdot |\Lambda^n| \cdot K^j.
\]
Since $K < 1$, the right-hand side goes to zero as $j \to \infty$.
\end{proof}

\section{A formula for KMS states on the Cuntz-Krieger algebra}\label{sec:injectivity}

The next step in our proof of Theorem~\ref{thm:main} is to establish a formula for a KMS
state $\phi$ of $C^*(\Lambda)$ in terms of $\phi \circ \pi_U$. We will use this later to
show that $\pi^*_U$ is a continuous affine injection from KMS$_1$ states of
$(C^*(\Lambda), \alpha)$ to states of $C^*(\Per\Lambda)$.

\begin{thm}\label{thm:CK states}
Let $\Lambda$ be a strongly connected finite $k$-graph, let $x^\Lambda$ be the unimodular
Perron-Frobenius eigenvector of $\Lambda$, and let $\alpha$ be the preferred dynamics on
$C^*(\Lambda)$. Let $U : \Per\Lambda \to C^*(\Lambda)$ be the unitary representation
$m-n\mapsto \sum_{\mu\in\Lambda^m}s_\mu s_{\theta_{m,n}(\mu)}^*$ of
Proposition~\ref{prp:unitary group}. If $\phi$ is a KMS$_1$ state for
$(C^*(\Lambda),\alpha)$, then
\begin{equation}\label{eq:CK phi formula}
\phi(s_\mu s^*_\nu)
    = \begin{cases}
        \rho(\Lambda)^{-d(\mu)} x^\Lambda_{s(\mu)} \phi(U_{d(\mu) - d(\nu)})
            &\text{ if $d(\mu) - d(\nu) \in \Per\Lambda$}\\
            &\qquad\text{ and $\theta_{d(\mu), d(\nu)}(\mu) =  \nu$} \\
        0   &\text{ otherwise.}
    \end{cases}
\end{equation}
\end{thm}

\begin{rmk}
On the face of it, the formula~\eqref{eq:CK phi formula} doesn't appear to satisfy
$\phi(s_\mu s^*_\nu) = \overline{\phi(s_\nu s^*_\mu)}$ (as a state must) because the
coefficient $\rho(\Lambda)^{-d(\mu)}$ doesn't appear to be symmetric in $\mu$ and $\nu$.
But all is well: \eqref{eq:rationally dependent} shows that $\rho(\Lambda)^{-d(\mu)} =
\rho(\Lambda)^{-d(\nu)}$ for $d(\mu) - d(\nu) \in \Per\Lambda$.
\end{rmk}

Our proof of Theorem~\ref{thm:CK states} requires a preliminary lemma.

\begin{lem}\label{lem:CK bound}
Let $\Lambda$ be a strongly connected finite $k$-graph. Let $\alpha$ be the preferred
dynamics on $C^*(\Lambda)$. Suppose that $\phi$ is a KMS$_1$ state for $(C^*(\Lambda),
\alpha)$. Let $\mu,\nu \in \Lambda$ with $s(\mu) = s(\nu)$. Then for every $p \in \NN^k$
we have
\[
    |\phi(s_\mu s^*_\nu)| \le
        \sum_{\substack{\lambda \in s(\mu)\Lambda^p\\ \Lambda^{\min}(\mu\lambda,\nu\lambda) \not= \emptyset}}
                                \phi(s_{\mu\lambda} s^*_{\mu\lambda}).
\]
\end{lem}

\begin{proof}
First suppose that $\rho(\Lambda)^{d(\mu)} \not= \rho(\Lambda)^{d(\nu)}$. Applying the
KMS condition twice, as in the end of the proof of
\cite[Proposition~3.1\,(b)]{aHLRSk-graphs}, gives $\phi(s_\mu s^*_\nu) =
\rho(\Lambda)^{d(\mu) - d(\nu)} \phi(s_\mu s^*_\nu)$. Hence $\phi(s_\mu s^*_\nu) = 0$,
and the result is trivial.

Second suppose that $\rho(\Lambda)^{d(\mu)} = \rho(\Lambda)^{d(\nu)}$. Applying~(CK4) and
the triangle inequality gives $|\phi(s_\mu s^*_\nu)| \le \sum_{\lambda \in
s(\mu)\Lambda^p} |\phi(s_{\mu\lambda} s^*_{\nu\lambda})|$. As in the proof of
\cite[Lemma~5.3\,(a)]{aHLRSk-graphs}, the KMS condition combined with the relation
$s^*_\eta s_\zeta = \sum_{(\alpha,\beta) \in \Lambda^{\min}(\eta,\zeta)} s_\alpha
s^*_\beta$ shows that $\phi(s_{\mu\lambda} s^*_{\nu\lambda}) = 0$ whenever
$\Lambda^{\min}(\mu\lambda, \nu\lambda) = \emptyset$. Hence
\[
\sum_{\lambda \in s(\mu)\Lambda^p} |\phi(s_{\mu\lambda} s^*_{\nu\lambda})|
    = \sum_{\substack{\lambda \in s(\mu)\Lambda^p\\ \Lambda^{\min}(\mu\lambda,\nu\lambda) \not= \emptyset}}
        |\phi(s_{\mu\lambda} s^*_{\nu\lambda})|.
\]
Since $\rho(\Lambda)^{d(\mu)} = \rho(\Lambda)^{d(\nu)}$, an argument using the
Cauchy-Schwarz inequality (see \cite[Lemma~5.2]{aHLRSk-graphs}) shows that each
$|\phi(s_{\mu\lambda} s^*_{\nu\lambda})| \le \phi(s_{\mu\lambda} s^*_{\mu\lambda})$, and
the result follows.
\end{proof}

\begin{proof}[Proof of Theorem~\ref{thm:CK states}]
Let $M$ be the measure on $\Lambda^\infty$ obtained in Proposition~\ref{prop:measure
e-u}, so that $\phi(s_\lambda s^*_\lambda) = M(Z(\lambda))$ for all $\lambda \in
\Lambda$.

First suppose that $d(\mu) - d(\nu) \not\in \Per\Lambda$. Choose $a \in \NN^k$ and $0 < K
< 1$ as in Lemma~\ref{lem:measure estimate}. For $j \in \NN$, Lemma~\ref{lem:CK bound}
implies that
\[
|\phi(s_\mu s^*_\nu)|
    \le \sum_{\substack{\lambda \in s(\mu)\Lambda^{ja}\\ \Lambda^{\min}(\mu\lambda,\nu\lambda) \not= \emptyset}}
                         \phi(s_{\mu\lambda} s^*_{\mu\lambda})
    = M\bigg(\bigcup_{\substack{\lambda \in s(\mu)\Lambda^{ja}\\ \Lambda^{\min}(\mu\lambda,\nu\lambda) \not= \emptyset}}
                         Z(\mu\lambda)\bigg).
\]
By choice of $K$ and $a$, the right-hand side is dominated by $K^j M(Z(\mu))$. This goes
to zero as $j \to \infty$, and so $\phi(s_\mu s^*_\nu) = 0$.

Now suppose that $\mu \in \Lambda^m$ and $\nu \in \Lambda^n$ with $m-n \in \Per\Lambda$.
We start by showing that
\begin{equation}\label{eq:red to source}
    \phi(s_\mu s^*_\nu) = \delta_{\theta_{m,n}(\mu), \nu} \rho(\Lambda)^{-m} \phi(p_{s(\mu)} U_{m-n}).
\end{equation}
Lemma~\ref{lem:per->bijection} implies that $s_\mu s^*_\mu = s_{\theta_{m,n}(\mu)}
s^*_{\theta_{m,n}(\mu)}$, and so the KMS condition implies that
\[
\phi(s_\mu s^*_\nu)
    = \phi(s_{\theta_{m,n}(\mu)} s^*_{\theta_{m,n}(\mu)} s_\mu s^*_\nu)
    = \phi(s_\mu s^*_\nu s_{\theta_{m,n}(\mu)} s^*_{\theta_{m,n}(\mu)})
    =\delta_{\nu, \theta_{m,n}(\mu)}\phi(s_\mu s^*_{\theta_{m,n}(\mu)})
\]
since $d(\nu) = d(\theta_{m,n}(\mu))$. The KMS condition gives
\[
\phi(s_\mu s^*_{\theta_{m,n}(\mu)})
    = \rho(\Lambda)^{-m} \phi(s^*_{\theta_{m,n}(\mu)} s_\mu)
    = \rho(\Lambda)^{-m} \phi\Big(\sum_{(\alpha,\beta) \in \Lambda^{\min}(\theta_{m,n}(\mu),\mu)} s_\alpha s^*_\beta\Big).
\]
Let $p := (m \vee n) - m$ and $q := (m \vee n) - n$. Corollary~\ref{cor:lmin<->theta}
implies that $\Lambda^{\min}(\theta_{m,n}(\mu), \mu) = \{(\alpha, \theta_{q,p}(\alpha) :
\alpha \in s(\mu)\Lambda^q\}$. Hence
\begin{align*}
\phi(s_\mu s^*_{\theta_{m,n}(\mu)})
    &= \rho(\Lambda)^{-m} \phi\Big(\sum_{\alpha \in s(\mu)\Lambda^q} s_\alpha s^*_{\theta_{q,p}(\alpha)}\Big)\\
    &= \rho(\Lambda)^{-m} \phi(p_{s(\mu)}U_{q-p})
    =\rho(\Lambda)^{-m}\phi(p_{s(\mu)} U_{m-n})
\end{align*}
since $q - p = m - n$. This gives~\eqref{eq:red to source}.

To establish~\eqref{eq:CK phi formula}, it now suffices to show that $\phi(p_v U_{n-m}) =
x^\Lambda_v \phi(U_{n-m})$ for all $v \in \Lambda^0$. To see this, consider the vector
$(y^{n-m}_v) \in \CC^{\Lambda^0}$ defined by $y^{n-m}_v = \phi(p_v U_{n-m})$. Fix $1 \le
i \le k$ and $v \in \Lambda^0$. Proposition~\ref{prp:unitary group} implies that
$U_{n-m}$ is central in $C^*(\Lambda)$. Using this and the Cuntz-Krieger relation and
then the KMS condition, we calculate:
\begin{align*}
y^{n-m}_v
    &= \phi(p_v U_{n-m})
    = \sum_{\lambda \in v\Lambda^{e_i}} \phi(s_\lambda s^*_\lambda U_{n-m})
    = \sum_{\lambda \in v\Lambda^{e_i}} \phi(s_\lambda U_{n-m} s^*_\lambda)\\
    &= \sum_{\lambda \in v\Lambda^{e_i}} \rho(\Lambda)_i^{-1} \phi(p_{s(\lambda)} U_{n-m})
    = \rho(A_i)^{-1} \sum_{w \in \Lambda^0} A_i(v,w) y^{n-m}_w
    = \rho(A_i)^{-1} (A_i y^{n-m})_v.
\end{align*}
Hence $y^{n-m}$ is an eigenvector of each $A_i$ with eigenvalue $\rho(A_i)$.
Corollary~\ref{cor-PFgraph}(\ref{cor-PFgraph2}) now implies that $y^{n-m} = zx^\Lambda$
for some $z \in \CC$. Since $x^\Lambda$ has unit $1$-norm, we have
\[
z = \sum_{v \in \Lambda^0} z x^\Lambda_v
    = \sum_{v \in \Lambda^0} y^{n-m}_v
    = \phi\Big(\sum_{v \in \Lambda^0} p_v U_{n-m}\Big)
    = \phi(U_{n-m}).\qedhere
\]
\end{proof}

\section{Constructing KMS states on the Cuntz-Krieger algebra}\label{sec:surjectivity}

In this section we construct a KMS$_1$ state $\phi_1$ of $(C^*(\Lambda), \alpha)$ such
that $\pi^*_U\phi_1$ is the identity character of $C^*(\Per\Lambda)$. We then show that
every character of $C^*(\Per\Lambda)$ is obtained by composing $\pi^*_U \phi_1$ with a
gauge automorphism $\gamma_z$. At the end of the section we combine this with
Theorem~\ref{thm:CK states} to prove our main theorem.

Let $\{h_x : x \in \Lambda^\infty\}$ be the orthonormal basis of point masses in
$\ell^2(\Lambda^\infty)$. Recall from the proof of \cite[Proposition~2.11]{KP} that there
is a Cuntz-Krieger $\Lambda$-family $\{S_\lambda : \lambda \in \Lambda\}$ in
$\Bb(\ell^2(\Lambda^\infty))$ such that $S_\lambda h_x = \delta_{s(\lambda), r(x)}
h_{\lambda x}$. We then have $S^*_\lambda h_x = \delta_{\lambda, x(0, d(\lambda))}
h_{\sigma^{d(\lambda)}(x)}$. The universal property of $C^*(\Lambda)$ implies that there
is a representation $\pi_S : C^*(\Lambda) \to \Bb(\ell^2(\Lambda^\infty))$ such that
$\pi_S(s_\lambda) = S_\lambda$. We call $\pi_S$ the \emph{infinite-path representation}.

\begin{lem}\label{lem:measurable}
Let $\Lambda$ be a strongly connected finite $k$-graph, and let $M$ be the measure on
$\Lambda^\infty$ obtained in Proposition~\ref{prop:measure e-u}.
\begin{enumerate}
\item\label{item-alem:measurable} Let $\mu,\nu \in \Lambda$. Then
\begin{align*}
M\big(\{x \in \Lambda^\infty &: x=\mu y=\nu y\text{\ for some $y\in\Lambda^\infty$}\}\big)\\
&=
\begin{cases}
M(Z(\mu))&\text{if $d(\mu) - d(\nu) \in \Per\Lambda$ and $\theta_{d(\mu),d(\nu)}(\mu) = \nu$}\\
0&otherwise.
\end{cases}
\end{align*}
\item\label{item-blem:measurable} Let $\pi_S$ be the infinite-path representation.
    For $a \in C^*(\Lambda)$, the function $f_a : x\mapsto \big(\pi_S(a) h_x \mid
    h_x\big)$ is $M$-integrable and
    \[
    \Big|\int_{\Lambda^\infty} \big(\pi_S(a) h_x \mid h_x\big)\,dM(x)\Big| \le\|a\|.
    \]
\end{enumerate}
\end{lem}

\begin{proof}
For convenience, write $Z_{\mu,\nu}:=\{x \in \Lambda^\infty : x=\mu y=\nu y\text{\ for
some $y\in\Lambda^\infty$}\}$. Since $Z_{\mu,\nu}$ is closed it is measurable.

First  suppose  that $d(\mu) - d(\nu) \not\in \Per\Lambda$.  Then $M(Z_{\mu,\nu})\leq
M(\{x\in\Lambda^\infty: \sigma^{d(\mu)}(x)=\sigma^{d(\nu)}(x)\} = 0$ by
Proposition~\ref{prp:m sees Per}. Thus $M(Z_{\mu,\nu})=0$.

Second, suppose that $d(\mu) - d(\nu) \in \Per\Lambda$ and $\theta_{d(\mu), d(\nu)}(\mu)
\not= \nu$. Since $Z_{\mu,\nu}\subseteq Z(\mu)\cap Z(\nu)$, we deduce that
$Z_{\mu,\nu}=\emptyset$, and $M(Z_{\mu,\nu})=0$.

Third, suppose that $d(\mu) - d(\nu) \in \Per\Lambda$ and $\theta_{d(\mu), d(\nu)}(\mu) =
\nu$.  If $x\in Z(\mu)$, then $y = \sigma^{d(\mu)}(x)$ satisfies $x=\mu y$. So $\mu y=\nu
y$ by definition of $\theta_{d(\mu), d(\nu)}$, giving $x\in Z_{\mu, \nu}$. Thus
$Z_{\mu,\nu}= Z(\mu)$ and $M(Z_{\mu,\nu})=M(Z(\mu))$. This
gives~(\ref{item-alem:measurable}).

For~(\ref{item-blem:measurable}), observe that
\[
(\pi_S(s_\mu s^*_\nu) h_x \mid h_x)
    = (S^*_\nu h_x \mid S^*_\mu h_x)
    = \begin{cases}
        1 &\text{ if $x = \mu y = \nu y$ for some $y \in \Lambda^\infty$} \\
        0 &\text{ otherwise.}
    \end{cases}
\]
Hence $f_{s_\mu s^*_\nu}$ is the characteristic function of the measurable set
$Z_{\mu,\nu}$. Choose a sequence $a_n$ of finite linear combinations of the $s_\mu
s^*_\nu$ such that $a_n \to a$. Then each $f_{a_n}$ is a simple function. Continuity of
$\pi_S$ and of the inner product implies that $f_{a_n} \to f_a$ pointwise on
$\Lambda^\infty$. Thus $f_a$ is measurable. Finally,
\[
\textstyle\big|\int_{\Lambda^\infty} f_a(x)\,dM(x)\big|
    \le \int_{\Lambda^\infty} \big|(\pi_S(a) h_x \mid h_x)\big|\,dM(x)
    \le \int_{\Lambda^\infty} \|a\|\,dM(x)
    = \|a\|.\qedhere
\]
\end{proof}

\begin{prop}\label{prp:integral KMS}
Let $\Lambda$ be a strongly connected finite $k$-graph, and let $M$ be the measure on
$\Lambda^\infty$ obtained in Proposition~\ref{prop:measure e-u}. Let $\alpha$ be the
preferred dynamics on $C^*(\Lambda)$. Let $\pi_S$ be the infinite-path representation.
Then there is a KMS$_1$ state $\phi$ of $(C^*(\Lambda), \alpha)$ with formula
\begin{equation}\label{eq:phi(a)}\textstyle
    \phi(a) := \int_{\Lambda^\infty} \big(\pi_S(a) h_x \mid h_x\big)\,dM(x)\text{\ for  $a \in C^*(\Lambda)$.}
\end{equation}
\end{prop}
\begin{proof}
Lemma~\ref{lem:measurable} implies that~\eqref{eq:phi(a)} defines a norm-decreasing map
$\phi : C^*(\Lambda) \to \CC$. This $\phi$ is  linear and positive. It is a state because
\[\textstyle
\phi(1) = \int_{\Lambda^\infty}\big(\pi_S(1) h_x \mid h_x\big)\,dM(x)
    = \int_{\Lambda^\infty} \|h_x\|^2\,dM(x) = 1.
\]
It remains to verify the KMS condition. Unfortunately
\cite[Proposition~3.1(b)]{aHLRSk-graphs} does not apply since the coordinates of
$\rho(\Lambda)$ may not be rationally independent; indeed KMS states may not be supported
on the diagonal subalgebra.  So we have to check the KMS condition from first principles.

Suppose that $s(\mu) = s(\nu)$ and $s(\eta) = s(\zeta)$. We must show that
\begin{equation}\label{eq:KMS to show}
\phi(s_\mu s^*_\nu s_\eta s^*_\zeta)
    = \rho(\Lambda)^{-(d(\mu) - d(\nu))} \phi(s_\eta s^*_\zeta s_\mu s^*_\nu).
\end{equation}

Suppose first that $d(\mu) - d(\nu) + d(\eta) - d(\zeta) \not\in \Per\Lambda$.
Applying~(CK4), we obtain
\[
s_\mu s^*_\nu s_\eta s^*_\zeta
    = \sum_{\xi \in s(\nu)\Lambda^{d(\eta)}} \sum_{\omega \in s(\eta) \Lambda^{d(\nu)}}
        s_{\mu\xi} s^*_{\nu\xi} s_{\eta\omega} s^*_{\zeta\omega}
    = \sum_{\nu\xi = \eta\omega \in \Lambda^{d(\nu) + d(\eta)}} s_{\mu\xi} s^*_{\zeta\omega}.
\]
Each $d(\mu\xi) - d(\zeta\omega) = d(\mu) - d(\nu) + d(\eta) - d(\zeta) \not\in
\Per\Lambda$, and so Lemma~\ref{lem:measurable} implies that $\phi(s_\mu s^*_\nu s_\eta
s^*_\zeta) = 0$. Symmetry gives $\phi(s_\eta s^*_\zeta s_\mu s^*_\nu) = 0$, so both sides
of~\eqref{eq:KMS to show} are zero.

Now suppose that  $d(\mu) - d(\nu) + d(\eta) - d(\zeta) \in \Per\Lambda$. Let $q = d(\mu)
\vee d(\nu)$. Then
\[
\phi(s_\mu s^*_\nu s_\eta s^*_\zeta)
    = \sum_{\kappa \in s(\eta)\Lambda^q} \phi(s_\mu s^*_\nu s_{\eta\kappa} s^*_{\zeta\kappa}),
\]
and
\[
\rho(\Lambda)^{-(d(\mu) - d(\nu))} \phi(s_\eta s^*_\zeta s_\mu s^*_\nu)
    = \sum_{\kappa \in s(\eta)\Lambda^q} \rho(\Lambda)^{-(d(\mu) - d(\nu))}
        \phi(s_{\eta\kappa} s^*_{\zeta\kappa} s_\mu s^*_\nu ).
\]
Thus it suffices to establish~\eqref{eq:KMS to show} under the additional hypothesis that
$d(\eta), d(\zeta) \ge d(\mu) \vee d(\nu)$. Then $d(\eta) \ge d(\nu)$, and we have
\begin{align*}
\phi(s_\mu s^*_\nu s_\eta s^*_\zeta)
    &=\begin{cases}
        \phi(s_{\mu\tau} s^*_\zeta) &\text{ if $\eta = \nu\tau$} \\
        0 &\text{ otherwise}
    \end{cases}\\
    &=\begin{cases}
        \int_{\Lambda^\infty} \big(S^*_\zeta h_x \mid S^*_{\mu\tau} h_x\big)\,dM(x) &\text{ if $\eta = \nu\tau$} \\
        0 &\text{ otherwise}
    \end{cases}\\
    &=\begin{cases}
        M\big(\{x \in \Lambda^\infty : x = \mu\tau y = \zeta y\text{ for some }y\}\big) &\text{ if $\eta = \nu\tau$} \\
        0 &\text{ otherwise.}
    \end{cases}
\end{align*}
If $\eta=\nu\tau$, then
\[
d(\mu\tau)-d(\zeta)=d(\mu)+d(\tau)-d(\nu)+d(\nu)-d(\zeta)
    = d(\mu) - d(\nu) + d(\eta) - d(\zeta)\in \Per\Lambda
\]
by assumption. Thus Lemma~\ref{lem:measurable} gives
\[
\phi(s_\mu s^*_\nu s_\eta s^*_\zeta)
    =\begin{cases}
        M(Z(\mu\tau) )&\text{if $\eta = \nu\tau$, $d(\mu\tau)-d(\zeta)\in\Per\Lambda$, $\theta_{d(\mu\tau), d(\zeta)}(\mu\tau)=\zeta$}\\
        0&\text{otherwise.}
    \end{cases}
\]
A similar argument gives
\[
\phi(s_\eta s^*_\zeta s_\mu s^*_\nu)
    =\begin{cases}
        M(Z(\eta) )&\text{if $\zeta = \mu\beta$, $d(\eta)-d(\nu\beta)\in\Per\Lambda$, $\theta_{d(\eta), d(\nu\beta)}(\eta)=\nu\beta$}\\
        0&\text{otherwise.}
    \end{cases}
\]

We check that the conditions appearing in the right-hand sides of these expressions for
$\phi(s_\mu s^*_\nu s_\eta s^*_\zeta)$ and $\phi(s_\eta s^*_\zeta s_\mu s^*_\nu)$ match
up. Suppose that the three conditions of the first expression hold:
\begin{equation}\label{eq:mnez cond}
    \eta = \nu\tau,\quad d(\mu\tau)-d(\zeta)\in\Per\Lambda\quad \text{ and }\quad\theta_{d(\mu\tau), d(\zeta)}(\mu\tau)=\zeta.
\end{equation}
Then $d(\tau) - (d(\zeta) - d(\mu)) \in \Per\Lambda$. Let $\beta := \theta_{d(\tau),
d(\zeta) - d(\mu)}(\tau)$ (this makes sense since $d(\zeta) \ge d(\mu)$).
Proposition~\ref{prp:per group}(\ref{it:theta ext}) shows that
\[
\zeta
    = \theta_{d(\mu\tau), d(\zeta)}(\mu\tau)
    = \theta_{d(\mu) + d(\tau), d(\mu) + (d(\zeta) - d(\mu))}(\mu\tau)
    = \mu\beta.
\]
We have
\[
d(\nu\beta) - d(\eta)
    = d(\nu\tau) - d(\tau) + d(\beta) - d(\eta) = d(\beta) - d(\tau)
    = d(\zeta) - d(\mu\tau)
    \in \Per\Lambda.
\]
Proposition~\ref{prp:per group}(\ref{it:theta ext}) then gives $\theta_{d(\nu\beta),
d(\eta)}(\nu\beta) = \nu\theta_{d(\beta), d(\tau)}(\beta)$, and by
Lemma~\ref{lem:per->bijection}(\ref{it:theta prp}) this is $\nu\theta^{-1}_{d(\tau),
d(\beta)}(\beta) = \nu\tau$, which equals  $\eta$ by assumption. Another application of
Lemma~\ref{lem:per->bijection}(\ref{it:theta prp}) yields $\nu\beta = \theta_{d(\eta),
d(\nu\beta)}(\eta)$. So the three conditions of the second expression hold:
\begin{equation}\label{eq:ezmn cond}
    \zeta = \mu\beta,\quad d(\eta)-d(\nu\beta)\in\Per\Lambda\quad \text{ and }\quad\theta_{d(\eta), d(\nu\beta)}(\eta)=\nu\beta.
\end{equation}
A symmetric argument shows that \eqref{eq:ezmn cond} implies \eqref{eq:mnez cond}.

To establish~\eqref{eq:KMS to show}, first suppose that~\eqref{eq:mnez cond} fails. Then
so does~\eqref{eq:ezmn cond}, and both sides of~\eqref{eq:KMS to show} are zero. Now
suppose that~\eqref{eq:mnez cond} holds. Then so does~\eqref{eq:ezmn cond}, and so
\begin{flalign*}
&&\phi(s_\mu s^*_\nu s_\eta s^*_\zeta)
    &= M(Z(\mu\tau))
    = \rho(\Lambda)^{-d(\mu\tau)} M(Z(s(\tau))) &\\
    &&&= \rho(\Lambda)^{d(\eta) - d(\mu\tau)} M(Z(\eta))
    = \rho(\Lambda)^{-(d(\mu) - d(\nu))} \phi(s_\eta s^*_\zeta s_\mu s^*_\nu).&\qedhere
\end{flalign*}
\end{proof}

Since the KMS$_1$ state $\phi$ of Proposition~\ref{prp:integral KMS} may not be supported
on $\clsp\{s_\lambda s^*_\lambda\}$ we can now perturb by gauge automorphisms $\gamma_z$
to obtain new KMS$_1$ states.

\begin{cor}\label{cor:many KMS}
Suppose $\Lambda$ is a strongly connected finite $k$-graph. Let $x^\Lambda$ be the
unimodular Perron-Frobenius eigenvector of $\Lambda$. Let $\alpha$ be the preferred
dynamics on $C^*(\Lambda)$. For each $z \in \TT^k$ there is a KMS$_1$ state $\phi_z$ of
$(C^*(\Lambda), \alpha)$ satisfying
\begin{equation}\label{eq:phiz formula}
\phi_z(s_\mu s^*_\nu)
    = \begin{cases}
        \rho(\Lambda)^{-d(\mu)} z^{d(\mu) - d(\nu)} x^\Lambda_{s(\mu)}
            &\text{ if $d(\mu) - d(\nu) \in \Per\Lambda$}\\
            &\qquad\text{ and $\theta_{d(\mu), d(\nu)}(\mu) =  \nu$} \\
        0   &\text{ otherwise.}
    \end{cases}
\end{equation}
\end{cor}
\begin{proof}
Let $\phi$ be the KMS$_1$ state of Proposition~\ref{prp:integral KMS}. Let  $z\in\TT^k$
and $\phi_z = \phi \circ \gamma_z$. Then $\phi_z$ is a state. Using the KMS condition for
$\phi$, we calculate:
\begin{align*}
\phi_z(s_\mu s^*_\nu s_\eta s^*_\zeta)
    &= z^{d(\mu) - d(\nu) + d(\eta) - d(\zeta)} \phi(s_\mu s^*_\nu s_\eta s^*_\zeta) \\
    &= z^{d(\mu) - d(\nu) + d(\eta) - d(\zeta)} \rho(\Lambda)^{-(d(\mu) - d(\nu))} \phi(s_\eta s^*_\zeta s_\mu s^*_\nu) \\
    &= \rho(\Lambda)^{-(d(\mu) - d(\nu))} \phi_z(s_\eta s^*_\zeta s_\mu s^*_\nu).
\end{align*}
Hence $\phi_z$ is a KMS$_1$ state of $(C^*(\Lambda), \alpha)$.

Let $\mu,\nu \in \Lambda$  and let $M$ be the measure on $\Lambda^\infty$ obtained in
Proposition~\ref{prop:measure e-u}. The formula for $\phi$ in
Proposition~\ref{prp:integral KMS} gives
\begin{align*}
\phi_z(s_\mu s^*_\nu)
    &=\textstyle \int_{\Lambda^\infty} \big(\pi_S(\gamma_z(s_\mu s^*_\nu)) h_x \mid h_x \big)\,dM(x)
    = z^{d(\mu) - d(\nu)} \int_{\Lambda^\infty} \big(\pi_S(s_\mu s^*_\nu) h_x \mid h_x \big)\,dM(x).\\
\intertext{
By Lemma~\ref{lem:measurable}, this is}
&= \begin{cases}
        z^{d(\mu) - d(\nu)} M(Z(\mu)) &\text{ if $d(\mu) - d(\nu) \in \Per\Lambda$ and $\theta_{d(\mu), d(\nu)}(\mu) = \nu$}\\
        0 &\text{ otherwise.}
    \end{cases}
\end{align*}
Now~\eqref{eq:phiz formula} follows from~\eqref{eq:cylinder measure}.
\end{proof}

\begin{proof}[Proof of Theorem~\ref{thm:main}]
It is clear that $\phi \mapsto \phi \circ \pi_U$ is continuous and affine. To see that it
is injective, suppose that $\phi$ and $\phi'$ are KMS states of $C^*(\Lambda)$ such that
$\phi \circ \pi_U = \phi' \circ \pi_U$. Then the formula~\eqref{eq:CK phi formula}
implies that $\phi(s_\mu s^*_\nu) = \phi'(s_\mu s^*_\nu)$ for all $\mu,\nu$, and so $\phi
= \phi'$.

To prove that $\pi^*_U$ is surjective, we first show that every pure state of
$C^*(\Per\Lambda)$ belongs to the range of $\pi_U^*$. Fix a pure state $\chi$ of
$C^*(\Per\Lambda)$. Since $C^*(\Per\Lambda)$ is commutative, $\chi$ is a 1-dimensional
representation and hence determines a character, also denoted $\chi$, of $\Per\Lambda$.
Choose $z \in \TT^k$ such that $z^m = \chi(m)$ for all $m \in \Per\Lambda$. Let $\phi_z$
be the KMS$_1$ state of Corollary~\ref{cor:many KMS}. Let $i_{\Per\Lambda} : \Per\Lambda
\to C^*(\Per\Lambda)$ be the universal unitary representation. For $m-n \in \Per\Lambda$,
we have
\[
\phi_z \circ \pi_U(i_{\Per\Lambda}(m-n))
    = \phi_z(U_{m-n})
    = \sum_{\mu \in \Lambda^m} \phi_z(s_\mu s^*_{\theta_{m,n}(\mu)}).
\]
Applying the formula for $\phi_z$ from~\eqref{eq:phiz formula} to each term gives
\begin{align*}
\phi_z \circ \pi_U(i_{\Per\Lambda}(m-n))
&=\sum_{\mu \in \Lambda^m} \rho(\Lambda)^{-m}z^{m-n} x^\Lambda_{s(\mu)}
=\rho(\Lambda)^{-m}\chi(m-n)\sum_{\mu \in \Lambda^m}x^\Lambda_{s(\mu)}\\
&=\rho(\Lambda)^{-m}\chi(m-n)\sum_{v,w \in \Lambda^0}A^m(v,w)x^\Lambda_{w}\\
&=\rho(\Lambda)^{-m}\chi(m-n)\sum_{v \in \Lambda^0}(A^mx^\Lambda)_{v}\\
&=\rho(\Lambda)^{-m}\chi(m-n)\sum_{v \in \Lambda^0}\rho(\Lambda)^m x^\Lambda_v
= \chi(m-n).
\end{align*}
Hence $\chi = \phi_z \circ \pi_U = \pi_U^*(\phi_z)$.

Since $\pi_U^*$ is affine, every convex combination of pure states of $C^*(\Per\Lambda)$
is in the range of $\pi_U^*$. Now fix a state $\psi$ of $C^*(\Per\Lambda)$. The
Krein-Milman theorem implies that there is a sequence $(\psi_n)$ of convex combinations
of pure states of $C^*(\Per\Lambda)$ such that $\psi_n \to \psi$. Each $\psi_n$ is in the
range of $\pi_U^*$, so it suffices to show that the range of $\pi_U^*$ is closed. The
KMS$_1$ simplex of $(C^*(\Lambda), \alpha)$ is compact \cite[Theorem~5.3.30(1)]{BR}, and
so its image under the continuous map $\pi_U^*$ is also compact. Since the state space of
$C^*(\Per\Lambda)$ is Hausdorff, we deduce that the image of $\pi_U^*$ is closed.
\end{proof}

\begin{rmk}\label{rmk:inverse}
Theorem~\ref{thm:CK states} shows how to describe the inverse of $\pi_U^*$. Let $\psi$ be
a state of $C^*(\Per\Lambda)$. Then $\phi := (\pi^*_U)^{-1}(\psi)$ satisfies $\phi \circ
\pi_U = \psi$ and so $\phi(U_{m-n}) = \phi \circ \pi_U(i_{\Per\Lambda}(m-n)) =
\psi(i_{\Per\Lambda}(m-n))$. So~\eqref{eq:CK phi formula} shows that
\[
\phi(s_\mu s^*_\nu)
    = \begin{cases}
        \rho(\Lambda)^{-d(\mu)} \psi\big(i_{\Per\Lambda}(d(\mu) - d(\nu))\big) x^\Lambda_{s(\mu)}
            &\text{ if $d(\mu) - d(\nu) \in \Per\Lambda$}\\
            &\text{\qquad and $\theta_{d(\mu), d(\nu)}(\mu) =  \nu$} \\
        0   &\text{ otherwise.}
    \end{cases}
\]
\end{rmk}

\section{Consequences of our main theorem}\label{sec:consequences}

\subsection*{A question of Yang}

In \cite{Yang:IJM10,Yang-preprint2013}, Yang studies a particular state $\omega$ on the
$C^*$-algebra of a finite $k$-graph with one vertex. She asks whether this $\omega$ is a
factor state if and only if $\Lambda$ is aperiodic. We will use the following theorem to
give an affirmative answer for a much broader class of $k$-graphs. We explain precisely
how our theorem relates to Yang's conjecture in Remark~\ref{rmk:YangConnection}.

Given a state $\phi$ of a $C^*$-algebra $A$, we write $\pi_\phi$ for the associated GNS
representation of $A$. Recall that $\phi$ is a factor state if the double-commutant
$\pi_\phi(A)''$ is a factor.

\begin{thm}\label{thm:YangConnection}
Suppose that $\Lambda$ is a strongly connected finite $k$-graph. Let $\alpha$ be the
preferred dynamics on $C^*(\Lambda)$, and let $x^\Lambda$ be the unimodular
Perron-Frobenius eigenvector of $\Lambda$ (see Definition~\ref{PF-vector}). Let $\gamma$
denote the gauge action of $\TT^k$ on $C^*(\Lambda)$. There is a KMS$_1$ state $\omega$
of $(C^*(\Lambda), \alpha)$ such that
\begin{equation}\label{eq:omegadef}
\omega(s_\mu s^*_\nu)
    = \delta_{\mu,\nu} \rho(\Lambda)^{-d(\mu)} x^\Lambda_{s(\mu)}\quad\text{ for all $\mu,\nu$.}
\end{equation}
This $\omega$ is the unique $\gamma$-invariant KMS state of $(C^*(\Lambda), \alpha)$, and
restricts to a trace on the fixed-point algebra $C^*(\Lambda)^\gamma$. The following are
equivalent:
\begin{enumerate}
\item\label{it:aper} $\Lambda$ is aperiodic;
\item\label{it:simple} $C^*(\Lambda)$ is simple;
\item\label{it:factor} $\omega$ is a factor state;
\item\label{it:onlyKMS} $\omega$ is the only KMS$_1$ state of $(C^*(\Lambda),
    \alpha)$.
\end{enumerate}
\end{thm}
\begin{proof}
Let $\Tr$ be the trace on $C^*(\Per\Lambda)$ corresponding to Haar measure on
$(\Per\Lambda)\widehat{\;}$. Then $\Tr(i_{\Per\Lambda}(g)) = \delta_{g,0}$ for $g \in
\Per\Lambda$. Remark~\ref{rmk:inverse} shows that there is a KMS$_1$ state of
$(C^*(\Lambda), \alpha)$ satisfying
\[
\omega(s_\mu s^*_\nu)
    = \begin{cases}
        \rho(\Lambda)^{-d(\mu)} \Tr\big(i_{\Per\Lambda}(d(\mu) - d(\nu))\big) x^\Lambda_{s(\mu)}
            &\text{ if $d(\mu) - d(\nu) \in \Per\Lambda$}\\
            &\text{\qquad and $\theta_{d(\mu), d(\nu)}(\mu) =  \nu$} \\
        0   &\text{ otherwise,}
    \end{cases}
\]
and that $\omega \circ \pi_U = \Tr$. Lemma~\ref{lem:per->bijection}(\ref{it:theta prp})
shows that $\theta_{m,m} = \id_{\Lambda^m}$ for each $m \in \NN^k$, so $\omega$
satisfies~\eqref{eq:omegadef}.

The formula~\eqref{eq:omegadef} shows that $\omega(\gamma_z(s_\mu s^*_\nu)) =
\omega(s_\mu s^*_\nu)$ for all $\mu,\nu$, and so $\omega$ is gauge-invariant. For
uniqueness, suppose that $\omega'$ is a gauge-invariant KMS$_1$ state of $C^*(\Lambda,
\alpha)$. For $m,n \in \NN^k$ with $m - n \in \Per\Lambda$, and for $z \in \TT^k$, we
have
\[
\omega'(U_{m-n})
    = \omega'(\gamma_z(U_{m-n}))
    = \omega'\Big(\sum_{\mu \in \Lambda^m} \gamma_z(s_\mu s^*_{\theta_{m,n}(\mu)})\Big)
    = z^{m-n} \omega'(U_{m-n}).
\]
So if $\omega'(U_{m-n}) \not=0$ then $z^{m-n} = 1$ for all $z \in \TT^k$, forcing $m-n =
0$. Hence $\omega' \circ \pi_U = \Tr = \omega \circ \pi_U$, and Theorem~\ref{thm:main}
implies that $\omega' = \omega$.

Since the dynamics $\alpha$ is a subgroup of the gauge action, every element of
$C^*(\Lambda)^\gamma$ is fixed by $\alpha$. In particular, every element of
$C^*(\Lambda)^\gamma$ is analytic, and the KMS condition implies that each $\omega(ab) =
\omega(b\alpha_{i}(a)) = \omega(ba)$, so $\omega$ is a trace on $C^*(\Lambda)^\gamma$.

It remains to establish that the conditions (\ref{it:aper})--(\ref{it:onlyKMS}) are
equivalent. We will prove \mbox{(\ref{it:aper})${}\iff{}$(\ref{it:simple})}, then
\mbox{(\ref{it:aper})${}\iff{}$(\ref{it:onlyKMS})}, and then
\mbox{(\ref{it:factor})${}\iff{}$(\ref{it:onlyKMS})}.

For \mbox{(\ref{it:aper})${}\iff{}$(\ref{it:simple})}, observe that since $\Lambda$ is
strongly connected it is cofinal (see \cite[Definition~4.7]{KP}). So combining
\cite[Theorem~3.1]{RobertsonSims:BLMS07} and $\mbox{(iii)}\iff\mbox{(i)}$ of
\cite[Lemma~3.2]{RobertsonSims:BLMS07} shows that $C^*(\Lambda)$ is simple if and only if
$\Lambda$ is aperiodic.

For \mbox{(\ref{it:aper})${}\implies{}$(\ref{it:onlyKMS})}, observe that since $\Lambda$
is aperiodic, we have $\Per\Lambda = \{0\}$ by Proposition~\ref{prp:Per<->per}. So $\Tr$
is the unique state of $C^*(\Per\Lambda)$, and Theorem~\ref{thm:main} implies that
$\omega$ is the unique KMS$_1$ state of $(C^*(\Lambda), \alpha)$. For
\mbox{(\ref{it:onlyKMS})${}\implies{}$(\ref{it:aper})}, observe that if $\omega$ is the
only KMS$_1$ state of $(C^*(\Lambda), \alpha)$, then Theorem~\ref{thm:main} shows that
$\Tr := \omega \circ \pi_U$ is the only state of $C^*(\Per\Lambda)$ and hence
$\Per\Lambda = \{0\}$. So Proposition~\ref{prp:Per<->per} implies that $\Lambda$ is
aperiodic.

For \mbox{(\ref{it:factor})${}\iff{}$(\ref{it:onlyKMS})}, first recall that the pure
states of $C^*(\Per\Lambda)$ are the states obtained from integration against point-mass
measures on $(\Per\Lambda)\widehat{\;}$. Since $\Tr$ is obtained from integration against
Haar measure, we deduce that $\Tr$ is a pure state if and only if it is the only state of
$C^*(\Per\Lambda)$. So Theorem~\ref{thm:main} shows that $\omega$ is an extreme point of
the KMS$_1$ simplex of $C^*(\Lambda, \alpha)$ if and only if it is the unique KMS$_1$
state. Theorem~5.3.30(3) of~\cite{BR} implies that a KMS$_1$ state is a factor state if
and only if it is an extreme KMS$_1$ state, giving
\mbox{(\ref{it:factor})${}\iff{}$(\ref{it:onlyKMS})}.
\end{proof}

We now discuss how this result relates to Yang's work.

\begin{rmk}\label{rmk:YangConnection}
Let $\Lambda$ be a row-finite $k$-graph with one vertex. Then \cite[Lemma~3.2]{KP}
implies that $C^*(\Lambda)^\gamma$ is a UHF algebra, and so has a unique trace $\tau$.
Let $\Phi : C^*(\Lambda) \to C^*(\Lambda)^\gamma$ be the conditional expectation obtained
from averaging over $\gamma$ as on page~6 of~\cite{KP}. In \cite{Yang-preprint2013} (see
also \cite{Yang:IJM10, Yang:xx11}), Yang studies the state $\tau \circ \Phi$.

We claim that the gauge-invariant KMS$_1$ state $\omega$ described in
Theorem~\ref{thm:YangConnection} is equal to $\tau \circ \Phi$. To see this, observe that
the formula~\eqref{eq:omegadef} shows that $\omega = \omega|_{C^*(\Lambda)^\gamma} \circ
\Phi$. Theorem~\ref{thm:YangConnection} implies that $\omega|_{C^*(\Lambda)^\gamma}$ is a
trace. Since $\tau$ is the unique trace on $C^*(\Lambda)^\gamma$, we deduce that
$\omega|_{C^*(\Lambda)^\gamma} = \tau$ and hence that $\omega = \tau \circ \Phi$.

Since $\Lambda$ has one vertex, each $A_i$ is the $1 \times 1$ matrix
$(|\Lambda^{e_i}|)$. So $\rho(\Lambda)$ is the vector, denoted $m$ in
\cite{Yang-preprint2013}, with entries $|\Lambda^{e_i}|$. So Yang's formula
\cite[Equation~(2)]{Yang-preprint2013} for the modular automorphism group $\sigma$ of the
extension of $\omega$ to $\pi_\omega(C^*(\Lambda))''$ shows that $\sigma$ agrees with the
preferred dynamics $\alpha$ on $C^*(\Lambda)$. Consequently, restricting
Theorem~\ref{thm:YangConnection} to $k$-graphs with one vertex improves
\cite[Theorem~5.3]{Yang-preprint2013} by proving its first assertion without the
hypothesis that $\{n \in \ZZ^k : \rho(\Lambda)^n = 1\}$ has rank at most $1$. This
confirms, for strongly-connected $k$-graphs, the first part of the conjecture stated for
single-vertex $2$-graphs in \cite[Remark~5.5]{Yang:xx11}.
\end{rmk}

\subsection*{The phase change for the preferred dynamics on the Toeplitz algebra}

For KMS states for the gauge actions on the Toeplitz algebras of finite graphs
\cite{aHLRS1-graphs, aHLRS1-graphs2}, the phase-changes that occur with decreasing
inverse temperature are from larger to smaller KMS simplices. Here we show that for many
$k$-graphs, there is a phase change of a very different nature at the critical
temperature for the preferred dynamics. In general, all sorts of phase changes can happen
as inverse temperatures approach a critical one (see, for example, \cite{BEK}), but this
is the first time we have seen this phenomenon for graph algebras. Recall that a
$k$-graph is \emph{periodic} if it is not aperiodic.

\begin{cor}\label{cor:phasechange}
Suppose that $\Lambda$ is a strongly connected finite $k$-graph and that $\rho(\Lambda)_i
> 1$ for all $i$. Denote by $\alpha$ the preferred dynamics on $\Tt C^*(\Lambda)$. For $\beta \in \RR$, let
$E_\beta$ be the set of extreme points of the KMS$_\beta$ simplex of $(\Tt C^*(\Lambda),
\alpha)$. Then
\[
|E_\beta|
    = \begin{cases}
        |\Lambda^0| &\text{ if $\beta>1$}\\
        \infty &\text{ if $\beta=1$ and $\Lambda$ is periodic}\\
        1&\text{ if $\beta = 1$ and $\Lambda$ is aperiodic.}\\
        0 &\text{ if $\beta < 1$.}
\end{cases}
\]
\end{cor}
\begin{proof}
Suppose that $\beta > 1$. Then $\beta\ln\rho(A_i) > \ln\rho(A_i)$ for all $i$. Since
$\Lambda$ is strongly connected it has no sources by Lemma~\ref{lem-nosources}. Thus
\cite[Theorem~6.1(c)]{aHLRSk-graphs} applies and shows that $|E_\beta| = |\Lambda^0|$.

Now suppose that $\beta=1$. Then
Corollary~\ref{cor:aHLRS2-Cor4.3,4.4}(\ref{it:allFactor}) implies that the quotient map
from $\Tt C^*(\Lambda)$ to $C^*(\Lambda)$ induces a bijection between $E_1$ and the
extreme KMS$_1$ states of $(C^*(\Lambda), \alpha)$. Hence Theorem~\ref{thm:main} gives a
bijection from $E_1$ to the pure states of $C^*(\Per\Lambda)$. If $\Lambda$ is periodic,
then $\Per\Lambda$ is a nontrivial subgroup of $\ZZ^k$ by
Lemma~\ref{lem:per->bijection}(\ref{it:lp->Per}), and so has infinitely many pure states.
If $\Lambda$ is aperiodic, then $\Per\Lambda = \{0\}$, and so $C^*(\Per\Lambda)$ has a
unique state.

If $\beta < 1$, then Corollary~\ref{cor:aHLRS2-Cor4.3,4.4}(\ref{it:KMSToeplitz}) applied
with $r = \ln\rho(\Lambda)$ implies that $(\Tt C^*(\Lambda), \alpha)$ admits no KMS
states.
\end{proof}

\begin{ex}
It is easy to construct examples exhibiting the phase change to an infinite-dimensional
KMS$_1$ simplex described in Corollary~\ref{cor:phasechange}. To see this, consider a
finite directed graph $E$ whose vertex matrix $A_E$ is irreducible and satisfies
$\rho(A_E) > 1$. The path category $E^*$ is a $1$-graph. Define $f : \NN^2 \to \NN$ by
$f(m,n) = m+n$, and let $\Lambda$ be the pullback $2$-graph $f^*E^*$ of
\cite[Definition~1.9]{KP}. Then $\Lambda^0  = E^0 \times \{0\}$, and each
$(v,0)\Lambda^{e_i} (w,0) = vE^1 w \times \{e_i\}$. So $A_1 = A_2 = A_E$ is irreducible,
and so $\Lambda$ is strongly connected. Corollary~3.5(iii) of \cite{KP} shows that
$C^*(\Lambda) \cong C^*(E) \otimes C(\TT)$, which is not simple. So the equivalence
\mbox{(\ref{it:simple})${}\iff{}$(\ref{it:aper})} of Theorem~\ref{thm:YangConnection}
shows that $\Lambda$ is periodic.
\end{ex}

\subsection*{Symmetries of the KMS simplex}

We show next that the gauge action on $C^*(\Lambda)$ induces a free and transitive action
of $(\Per\Lambda)\widehat{\;}$ on the KMS$_1$ simplex of $C^*(\Lambda)$. Recall that
$(\Per\Lambda)^\perp$ denotes the collection of characters of $\ZZ^k$ which are
identically 1 on $\Per\Lambda$. Identifying $\widehat{\ZZ^k}$ with $\TT^k$, we have
\[
(\Per\Lambda)^\perp = \{z \in \TT^k : z^n = 1\text{ for all }n \in \Per\Lambda\}.
\]
There is a homomorphism $q : \TT^k \to(\Per\Lambda)\widehat{\;}$ such that $q(z)(g) =
z^g$, and $\ker q = (\Per\Lambda)^\perp$.

\begin{prop}
Let $\Lambda$ be a strongly connected finite $k$-graph.
\begin{enumerate}
\item\label{it:whenequal} For $z,w \in \TT^k$, the states $\phi_z$ and $\phi_w$ of
    Corollary~\ref{cor:many KMS} are equal if and only if $z\overline{w} \in
    (\Per\Lambda)^\perp$.
\item\label{it:quotient} There is a homeomorphism $h$ of $(\Per\Lambda)\widehat{\;}$
    onto the set $E$ of extreme points of the KMS$_1$ simplex of $C^*(\Lambda)$ such
    that $h(q(z)) = \phi_z$ for all $z \in \TT^k$.
\item\label{it:symmetries} The gauge action $\gamma$ induces a free and transitive
    action $\widetilde{\gamma}^*$ of $(\Per\Lambda)\widehat{\;}$ on $E$ such that
    $\widetilde{\gamma}^*_{\chi}(h(\rho)) = h(\chi\rho)$ for $\chi,\rho \in
    (\Per\Lambda)\widehat{\;}$.
\end{enumerate}
\end{prop}
\begin{proof}
(\ref{it:whenequal}) Suppose that $z\overline{w} \in( \Per\Lambda)^\perp$. Then
$z^{d(\mu) - d(\nu)} = w^{d(\mu) - d(\nu)}$ whenever $d(\mu) - d(\nu) \in \Per\Lambda$.
Hence~\eqref{eq:phiz formula} implies that $\phi_z = \phi_w$.

Now suppose that $z\overline{w} \not\in( \Per\Lambda)^\perp$. Take $m - n \in
\Per\Lambda$ with $(z\overline{w})^{m-n} \not= 1$, and let $\mu\in \Lambda^m$. Let
$x^\Lambda$ be the unimodular Perron-Frobenius eigenvector of $\Lambda$.
Corollary~\ref{cor-PFgraph}(\ref{cor-PFgraph1}) implies that $x^\Lambda_{s(\mu)} \not=
0$, and so $z^{m-n} \rho(\Lambda)^{-m} x^\Lambda_{s(\mu)} \not= w^{m-n}
\rho(\Lambda)^{-m} x^\Lambda_{s(\mu)}$. Hence~\eqref{eq:phiz formula} implies that
$\phi_z(s_\mu s^*_{\theta_{m,n}(\mu)}) \not= \phi_w(s_\mu s^*_{\theta_{m,n}(\mu)})$.

(\ref{it:quotient}) Part~(\ref{it:whenequal}) implies that the formula $h(q(z)) = \phi_z$
determines a well-defined bijection from $(\Per\Lambda)\widehat{\;}$ to $E$. Suppose that
$\chi_n \to \chi$ in $(\Per\Lambda)\widehat{\;}$, and choose $z_n \in \TT^k$ such that
$q(z_n) = \chi_n$. By passing to a subsequence we may assume that the $z_n$ converge to
some $z \in \TT^k$. We then have $q(z) = \chi$. The formula~\eqref{eq:phiz formula} shows
that $\phi_{z_n}(s_\mu s^*_\nu) \to \phi_z(s_\mu s^*_\nu)$ for all $\mu,\nu$, and an
$\varepsilon/3$-argument then shows that $\phi_{z_n} \to \phi_z$, and so $h(\chi_n) \to
h(\chi)$. Thus $h$ is a continuous bijection, and so a homeomorphism since
$(\Per\Lambda)\widehat{\;}\,$ is compact.

(\ref{it:symmetries}) Formula~\eqref{eq:phiz formula} implies that $\phi_w \circ \gamma_z
= \phi_{wz}$. So if $z'\bar{z} \in (\Per \Lambda)^\perp$, then part~(\ref{it:whenequal})
implies that $\phi_w \circ \gamma_z = \phi_w \circ \gamma_{z'}$ for all $w$. Hence the
action $\gamma^*$ of $\TT^k$ on $E$ induced by $\gamma$ descends to an action
$\widetilde{\gamma}^*$ of $(\Per\Lambda)\widehat{\;}$ satisfying
$\widetilde{\gamma}^*_{q(z)}\big(h(q(w))\big) = h\big(q(z)q(w)\big)$ as required. This
action is free and transitive because left translation in $(\Per\Lambda)\widehat{\;}$ is
free and transitive.
\end{proof}

\section{The groupoid model}\label{sec:groupoids}

In \cite{N}, Neshveyev studies KMS states for dynamics on groupoid $C^*$-algebras arising
from continuous $\RR$-valued cocycles on groupoids. The Cuntz-Krieger algebra of a
$k$-graph $\Lambda$ admits a groupoid model with such a dynamics, and in this section we
check that our Theorem~\ref{thm:main} agrees with Neshveyev's \cite[Theorem~1.3]{N} for
these examples.

\subsection*{Neshveyev's theorem}

Let $G$ be a locally compact second-countable \'etale groupoid and $c:G\to\RR$  a
continuous cocycle. There is a dynamics $\alpha^c$ on $C^*(G)$ such that
$\alpha_t^c(f)(g)=e^{itc(g)}f(g)$ for $f\in C_c(G)$ and $g\in G$.

Let $U$ be an open bisection of $G$ and write $T^U : r(U) \to s(U)$ for the homeomorphism
$r(g) \mapsto s(g)$ for $g \in U$. Recall from \cite[page~4]{N} that a measure $\mu$ on
$G^{(0)}$ is said to be quasi-invariant with Radon-Nikodym cocycle $e^{-\beta c}$ if
$\frac{d T^U_*\mu}{d\mu}(s(g)) = e^{-\beta c(g)}$ for every open bisection $U$ and every
$g \in U$.

For $x$ in the unit space $G^{(0)}$, write $G_x^x$ for the stability subgroup $\{g\in G:
r(g)=x=s(g)\}$ and $G_x$ for the subset $\{g\in G: s(g)=x\}$ of $G$. Theorem~1.3 of
\cite{N} describes the KMS$_\beta$ states of $(C^*(G),\alpha^c)$ in terms of pairs $(\mu,
\psi)$ consisting of a quasi-invariant probability measure $\mu$ on $G^{(0)}$ with
Radon-Nikodym cocycle $e^{- \beta c}$ and a $\mu$-measurable field $\psi = (\psi_x)_{x
\in G^{(0)}}$ of states $\psi_x : C^*(G^x_x) \to \CC$ such that for $\mu$-almost all $x
\in G^0$ we have
\begin{equation}\label{eq:invariance-again}
\psi_x(u_g) = \psi_{r(h)}(u_{hgh^{-1}})
    \text{ for all $g \in G^x_x$ and $h \in G_x$.}
\end{equation}
(There is a second condition which we can ignore because for non-zero inverse
temperatures $\beta$ the properties of $\mu$ ensure that it is always satisfied.)
Neshveyev's theorem does not distinguish between measurable fields that agree
$\mu$-almost everywhere.

\subsection*{The path groupoid} Let $\Lambda$ be a row-finite $k$-graph with no sources.
The set
\begin{equation*}
G:= \{(x,m-n,y) : x,y \in \Lambda^\infty, m,n\in\NN^k
    \text{ and } \sigma^m(x) = \sigma^n(y)\}
\end{equation*}
is a groupoid  with range and source maps $r(x,g,y) = (x,0,x)$, $s(x, g, y) = (y,0,y)$,
composition $(x,g,y)(y,h,z) = (x, g+h, z)$ and inverses $(x,g,y)^{-1} = (y,-g,x)$. We
identify $G^{(0)}$ with $\Lambda^\infty$ via $(x,0,x)\mapsto x$.

For $\lambda, \eta \in \Lambda$ with $s(\lambda) = s(\eta)$, define
\begin{align*}
Z(\lambda,\eta)
    = \{(x, d(\lambda) - d(\eta), y) \in G : x \in Z(\lambda), y \in Z(\eta)
                    \text{ and } \sigma^{d(\lambda)}(x) = \sigma^{d(\eta)}(y)\}.
\end{align*}
By Proposition~2.8 of \cite{KP}, the sets $Z(\lambda,\eta)$ form a basis for a locally
compact Hausdorff topology on $G$. With this topology $G$ is a second-countable \'etale
groupoid, called the \emph{path groupoid}. Each $Z(\eta,\lambda)$ is a compact open
bisection. By Corollary~3.5 of \cite{KP} there is an isomorphism of $C^*(\Lambda)$ onto
the $C^*$-algebra $C^*(G)$ of $G$ such that
\begin{equation}\label{eq-isoongen}
    t_\lambda\mapsto 1_{Z(\lambda, s(\lambda))}.
\end{equation}

\subsection*{Theorem~\ref{thm:main} and Neshveyev's theorem}
Let $\Lambda$ be a strongly connected finite  $k$-graph and let $G$ be its path groupoid.

There is a locally constant cocycle $c : G \to \RR$ given by $c(x, n, y) = n \cdot
\ln\rho(\Lambda)$. This cocycle induces a dynamics $\alpha^c: \RR \to \Aut C^*(G)$ such
that $\alpha^c_t(f)(x,n,y) = e^{itc(x,n,y)} f(x,n,y) = \rho(\Lambda)^{itn}$ for $f\in
C_c(G)$. It is straightforward to check that the  isomorphism of $C^*(\Lambda)$ onto
$C^*(G)$ characterised by \eqref{eq-isoongen} intertwines $\alpha^c$ and the preferred
dynamics $\alpha$ on $C^*(\Lambda)$.

It follows from  \eqref{eq:measure cond} that the measure $M$ on $\Lambda^\infty$ of
Proposition~\ref{prop:measure e-u} is quasi-invariant with Radon-Nikodym cocycle
$e^{-c}$; the next lemma implies that $M$ is the only such measure and  investigates its
support further.  For the latter, we note that if $g\in\Per\Lambda$, then there exist
$m,n\in\NN^k$ such that $g=m-n$ and $\sigma^m(x)=\sigma^n(x)$ for all
$x\in\Lambda^\infty$.  Thus for each $x\in\Lambda^\infty$,
\[\{x\}\times\Per\Lambda\times\{x\}\subseteq G_x^x=\{(x,g,x)\in G :x\in\Lambda^\infty\}.\]

\begin{lem}\label{lem:mu=M}
Suppose that $\mu$ is a non-zero quasi-invariant probability measure on $G^{(0)}=
\Lambda^\infty$ with Radon-Nikodym cocycle $e^{-c}$. Then $\mu$ is the measure $M$ of
Proposition~\ref{prop:measure e-u} and
\begin{equation}\label{eq:support mu}
M\big(\{x\in\Lambda^\infty: \{x\}\times\Per\Lambda\times\{x\}= G_x^x\}\big)=1.
\end{equation}
\end{lem}
\begin{proof}
Let $v\in\Lambda^0$ and $\lambda\in v\Lambda$.  Then $Z(\lambda,s(\lambda))$ is a
bisection  with $r\big(Z(\lambda, s(\lambda))\big)=Z(\lambda)$ and $s\big(Z(\lambda,
s(\lambda))\big)=Z(s(\lambda))$. By the quasi-invariance of $\mu$ we have
\begin{equation}\label{M and mu}
\mu(Z(\lambda)) = e^{-d(\lambda)\cdot \ln\rho(\Lambda)}\mu(Z(s(\lambda)) = \rho(\Lambda)^{-d(\lambda)}\mu(Z(s(\lambda)).
\end{equation}
In particular, if $\lambda\in v\Lambda^{e_i}$ then
$\mu(Z(\lambda))=\rho(A_i)^{-1}\mu(Z(s(\lambda))$. Thus
\begin{align}\label{eq:quasi-invariance-again}
\mu(Z(v))
    &\ge \mu\big(\bigsqcup_{w\in\Lambda^0}\bigsqcup_{\lambda \in v\Lambda^{e_i}w} Z(\lambda)\big)=\sum_{w\in\Lambda^0}\sum_{\lambda \in v\Lambda^{e_i}w} \mu\big(Z(\lambda)\big)\notag\\
    & =\rho(A_i)^{-1} \sum_{w\in\Lambda^0} A_i(v,w) \mu(Z(w)).
\end{align}
Set $m := \big(\mu(Z(v)\big) \in [0,\infty)^{\Lambda^0}$. Then
\eqref{eq:quasi-invariance-again} says that $m$ satisfies $\rho(A_i) m \ge A_im$. Also,
$\sum_{v\in\Lambda^0} m_v = \mu\big(\bigsqcup_{v\in\Lambda^0} Z(v)\big) =
\mu(\Lambda^\infty) = 1$. Thus Corollary~\ref{cor-PFgraph}\eqref{cor-PFgraph3} implies
that $m$ is the Perron-Frobenius eigenvector $x^\Lambda$ of $\Lambda$. Now~\eqref{M and
mu} shows that $\mu(Z(\lambda)) = \rho(\Lambda)^{-d(\lambda)} x^\Lambda_{s(\lambda)}$,
and this is $M(Z(\lambda))$ by~\eqref{eq:measure cond}. Since the $Z(\lambda)$ form a
basis for the topology on $\Lambda^\infty$ we have  $\mu =M$.

Finally,
\begin{align*}
\{x\in\Lambda^\infty : {}& \{x\}\times\Per\Lambda\times\{x\}= G_x^x\}\\
&=\{x\in\Lambda^\infty: m,n\in \NN^k\text{ and } \sigma^m(x)=\sigma^n(x)\Longrightarrow m-n\in\Per\Lambda\}\\
&=\bigcap_{m,n\in \NN^k} \{x\in\Lambda^\infty:m-n\notin\Per\Lambda \Longrightarrow \sigma^m(x)\neq\sigma^n(x)\}\\
&=\bigcap_{m,n\in \NN^k,m-n\notin \Per\Lambda} \{x\in\Lambda^\infty:\sigma^m(x)\neq\sigma^n(x)\}\\
&=\Lambda^\infty\setminus \bigcup_{m,n\in \NN^k,m-n\notin \Per\Lambda}\{x\in\Lambda^\infty:\sigma^m(x)=\sigma^n(x)\}.
\end{align*}
By Proposition~\ref{prp:m sees Per}, if $m-n\not\in\Per\Lambda$, then $M(\{x \in
\Lambda^\infty : \sigma^m(x) = \sigma^n(x)\}) = 0$. Since $\{(m,n) : m-n \not\in
\Per\Lambda\}$ is countable, this gives
\[
M\Big(\bigcup_{m,n\in \NN^k,m-n \not\in \Per\Lambda}
   \{x \in \Lambda^\infty : \sigma^m(x) = \sigma^n(x)\}\Big) = 0
\]
and~\eqref{eq:support mu} follows.
\end{proof}

Now let $(\mu, \psi)$ be one of  Neshveyev's pairs for $(C^*(G), \alpha^c)$. By
Lemma~\ref{lem:mu=M}, $\mu=M$ and $M\big(\{x\in\Lambda^\infty:
\{x\}\times\Per\Lambda\times\{x\}= G_x^x\}\big)=1$. Thus we may assume that $\psi_x=0$
unless $\{x\}\times\Per\Lambda\times\{x\}= G_x^x$. For each $x \in \Lambda^\infty$, let
$\iota_x: C^*(\Per\Lambda)\to C^*(\{x\}\times\Per\Lambda\times\{x\})$ be the isomorphism
such that $\iota_x(u_n) = u_{(x,n,x)}$. For $a\in C_c(\Per\Lambda)$, the
$M$-measurability of $\psi$ implies that  $x\mapsto \psi_x(\iota_x(a))$ is
$M$-measurable. Thus there is a state $\rho$ of $C^*(\Per\Lambda)$ such that
\begin{equation*}
	\rho(a) = \int_{\Lambda^\infty} \psi_x(\iota_x(a))\, dM(x).
\end{equation*}
for $a\in C^*(\Per\Lambda)$.

Conversely,  let $\rho$ be a state of $C^*(\Per\Lambda)$. The measure $M$ is
quasi-invariant with Radon-Nikodym cocycle $e^{-c}$. Define
\[
\rho_x =
    \begin{cases}
        \rho\circ\iota_x^{-1}&\text{ if $\{x\}\times\Per\Lambda\times\{x\}= G_x^x$}\\
        0&\text{else}.
    \end{cases}
\]
For $f\in  C_c(G)$, the map   $x\mapsto \sum_{m\in\Per\Lambda}
f(x,m,x)\psi_x(u_{(x,m,x)})=\sum_{m\in\Per\Lambda} f(x,m,x)\rho(u_m)$ is continuous,
hence measurable, and so $(\rho_x)$ is a measurable field.
Equation~\eqref{eq:invariance-again} follows because
$\rho_x(u_{(x,m,x)})=\rho(u_m)=\rho_y(u_{(y,m,y)})$. So $(M,(\rho_x))$ is one of
Neshveyev's pairs. Thus, reassuringly,  our Theorem~\ref{thm:main} and Neshveyev's
\cite[Theorem~1.3]{N} say the same. To prove Theorem~\ref{thm:main} using the groupoid
approach we would have had to do much the same work (except for
Proposition~\ref{prp:integral KMS}) and we would have lost the transparency of the direct
proof.

\end{document}